\documentclass[12pt,reqno]{amsart}
\usepackage[all]{xy}
\usepackage{amsmath}
\usepackage{amssymb}
\usepackage{latexsym}
\usepackage{enumerate}
\usepackage[active]{srcltx}
\usepackage{graphics}
\usepackage{epsfig}

\newtheorem{theorem}{Theorem}[section]
\newtheorem{proposition}[theorem]{Proposition}
\newtheorem{corollary}[theorem]{Corollary}
\newtheorem{lemma}[theorem]{Lemma}

\theoremstyle{definition}
\newtheorem{definition}[theorem]{Definition}
\newtheorem{example}[theorem]{Example}
\newtheorem{remark}[theorem]{Remark}
\newtheorem{point}[theorem]{}

\newcommand{\A}{{\mathcal A}}
\newcommand{\M}{{\mathcal M}}
\newcommand{\C}{{\mathcal C}}
\newcommand{\D}{{\mathcal D}}
\newcommand{\B}{{\mathcal B}}
\newcommand{\E}{{\mathcal E}}
\renewcommand{\L}{{\mathcal L}}
\newcommand{\K}{{\mathcal K}}
\newcommand{\T}{{\mathcal T}}
\newcommand{\W}{{\mathcal W}}
\newcommand{\Z}{{\mathbb Z}}

\newcommand{\I}{{\mathcal I}}
\newcommand{\ch}{{\mathrm{Ch}}}
\renewcommand{\O}{{\mathcal O}}

\newcommand{\Hom}{\mathrm{Hom}}

\newcommand{\downarrowright}[1]{\downarrow
\rlap{\raise0.1cm\hbox{$\scriptstyle{#1}$}}}

\newcommand{\downarrowleft}[1]{\rlap{\kern-0.2cm
\raise0.1cm\hbox{$\scriptstyle{#1}$}}\downarrow}

\newcommand{\uparrowright}[1]{\uparrow
\rlap{\lower0.1cm\hbox{$\scriptstyle{#1}$}}}
  
\newcommand{\uparrowleft}[1]{\rlap{\kern-0.2cm
\lower0.1cm\hbox{$\scriptstyle{#1}$}}\uparrow}

\newcommand{\ra}{\rightarrow}

\newcommand{\wmono}{\ar@{>->}[r]}

\def\umono{\ar@{_{(}->}[u]}
\def\uumono{\ar@{_{(}->}[uu]}

\def\lmono{\ar@{_{(}->}[l]}
\def\llmono{\ar@{_{(}->}[ll]}

\numberwithin{equation}{section}

\title[Relative homological algebra via truncations]{Relative homological algebra via truncations}

\author[W. Chach\'olski]{Wojciech Chach\'olski}
\author[A. Neeman]{Amnon Neeman}
\author[W. Pitsch]{Wolfgang Pitsch}
\author[J. Scherer]{J\'er\^{o}me Scherer}

\thanks{This research was partly supported by the Australian Research Council, Vetenskapsr{\aa}det and 	 G\"oran Gustafssons Stiftelse, and
FEDER/MEC grants MTM2013-42293 and MTM2016-80439-P}



\begin{document}

\begin{abstract}
To do homological algebra with unbounded chain complexes one needs to first find a way of constructing resolutions.
Spaltenstein solved this problem for chain complexes of $R$-modules by truncating further and further to the left, 
resolving the pieces, and gluing back the partial resolutions. Our aim is to give a homotopy theoretical interpretation
of this procedure, which may be extended to a relative setting. We work in an arbitrary abelian category  
$\mathcal{A}$ and fix a class of ``injective objects'' $\mathcal{I}$. We show that Spaltenstein's construction can be captured
by a pair of adjoint functors between unbounded chain complexes and towers of non-positively graded ones. This pair
of adjoint functors forms what we call a Quillen pair and the above process of truncations, partial resolutions, and gluing,
gives a meaningful way to resolve complexes in a relative setting \emph{up to a split error term}. In order to do homotopy
theory, and in particular to construct a well behaved relative derived category $D(\mathcal{A};\mathcal{I})$, we need more: the split error term must vanish.
This is the case when $\mathcal I$ is the class of all injective $R$-modules but not in general, not even for
certain classes of injectives modules over a Noetherian ring. The key property is a relative analogue of Roos's AB4*-$n$ axiom 
for abelian categories. Various concrete examples such as Gorenstein homological algebra and purity are also discussed.
\end{abstract}

\maketitle

\section*{Introduction}
Our aim in this work is to present a framework to do relative homological algebra. If homological algebra is understood as a means 
to study objects and functors in abelian categories through invariants determined by projective or injective resolutions, then relative
homological algebra should give us more flexibility in constructing resolutions, meaning we would like to be allowed to use a priori 
\emph{any}  object  as an injective. This idea goes back at least to Adamson \cite{MR0065546} for group
cohomology and Chevalley-Eilenberg \cite{MR0024908} for Lie algebra homology. Both were then subsumed in 
a general theory by Hochschild \cite{MR0080654}. The most complete reference for the classical point of view is 
Eilenberg--Moore \cite{MR0178036}.

Analogously, in homotopy theory one would traditionally use spheres to ``resolve spaces'' by constructing a CW-approximation, but
it has become very common nowadays to replace them by some other spaces and do $A$-homotopy theory, as developed for instance 
by Farjoun \cite{Farjoun95}. In fact, homotopical methods have already been applied to do relative homological algebra.
Christensen and Hovey \cite{MR1912401} show that, in many cases, one can equip the category of unbounded chain
complexes with a model category structure where the weak equivalences reflect a choice of new projective objects. It is
their work, and the relationship to Spaltenstein's explicit construction of a resolution for unbounded chain complex \cite{MR932640},
that motivated us originally. We wish to stress the point that, for us, it is as important to have a constructive method to build
relative resolutions as to know that there exists a formal method to invert certain relative quasi-isomorphisms (because
there is a relative model structure or a relative derived category for example).

More precisely we fix in an abelian category $\A$ a class $\I\subset\A$ of objects, called the \emph{relative injectives}, that will play the role of
usual injectives. 
This determines in turn two classes of maps: a class of \emph{relative monomorphisms} and 
a class of \emph{relative quasi-isomorphisms}. If $\I$ is the class of injective
objects these reduce to ordinary monomorphisms and ordinary quasi-isomorphisms.
Denote by $\ch(\A)$ the category of chain complexes over $\A$ and by $\W_\I$ the class of relative quasi-isomorphisms. Our aim is 
to construct the localized category $\ch(\A)[\W_\I]^{-1}$, in particular we would like to find a way to 
resolve chain complexes.

Disregarding set-theoretical problems, one could formally add inverses of the elements in $\W_\I$ to get $D(\A;\I) = \ch(\A)[\W_\I]^{-1}$.
With a little more care, for instance using the theory of null systems, one can construct $\ch(\A)[\W_\I]^{-1}$ by the calculus of fractions and 
endow it with a natural triangulated structure; this is done at the end of  Section~\ref{sec framework}. It is unwise though to completely 
disregard set-theoretic problems and Quillen devised in the late sixties the notion of a model category, see \cite{MR0223432},  which
provides a technique for overcoming this difficulty. 
On the category of left bounded chain complexes Bousfield \cite{MR2026537}
showed how to use Quillen's machine to construct 
the relative derived category $D_{\leq 0}(\A;\I)$.
An elementary exposition of Bousfield's relative model structure, including explicit methods to construct factorizations (and hence resolutions), 
can be found in Appendix~\ref{app left}. 

Our objective is to extend this construction and the model structure to unbounded complexes, but this is a more delicate issue, 
even  in the classical setting, see Spaltenstein~\cite{MR932640} or Serp\'e~\cite{serpe}. A relative model structure on $\ch(\A)$ 
would be nice, but we cannot apply homotopical localization techniques in a straightforward way since there is no obvious
\emph{set} of maps to invert. Anyway, we need less. Therefore we introduce a more flexible framework, 
namely that of a model approximation \cite{MR1879153}.
Our idea in this work is to approximate a complex by the tower of its truncations, just as Spaltenstein did. 
For this we observe first in Proposition~\ref{prop modelontowers} that a relative model structure on left bounded complexes 
induces a model structure on towers of left bounded complexes.
Diagrams of model categories have been studied by Greenlees and Shipley \cite{MR3254575} and play an important role in
equivariant stable homotopy theory, see for example \cite{BGKS}.
Recent work of Harpaz and Prasma \cite{MR3366868} proposes another viewpoint on such diagrams and model structures.

%
%

Second, we package the relationship between unbounded chain complexes and the category of towers $\text{\rm Tow}(\A,\I)$ equipped with
the relative model structure into what we call a Quillen pair. It consists of a pair of adjoint functors
\[
\text{\rm tow}:\text{\rm Ch}(\A)\rightleftarrows \text{\rm Tow}(\A,\I):\text{\rm lim}
\]
where the ``tower functor'' associates to a complex the tower given by truncating it further and further to the left, 
and the limit functor takes limits degreewise, see Proposition~\ref{prop rqp}. The left hand side is not a model category 
but its homotopical features
are reflected in the right hand side. To do homotopy theory with unbounded chain complexes we need this
Quillen pair to form a \emph{model approximation}, i.e.~to verify some extra compatibility condition of the adjoint pair
with resolutions, see Definition~\ref{def approx}. When this is the case resolutions of complexes are provided by an 
explicit recipe. Thus we need to understand when the Quillen pair is a model approximation. Our first answer concerns
rings with finite Krull dimension.


\medskip

\noindent
{\bf Theorem~\ref{theorem finite Kdim}.}
\emph{Let $R$ be a Noetherian ring of finite Krull dimension~$d$, and
$\I$ an injective class of injective modules. Then the category of towers
forms a model approximation for $\text{Ch}(R)$ equipped with
$\I$-equivalences.}

\medskip

When the Krull dimension is infinite it depends on the chosen class of injectives whether or not one can
resolve unbounded complexes by truncation. For Nagata's ring \cite{MR0155856} we construct in
Theorem~\ref{theorem counterexample}
an injective class $\I$
which fails to yield a model approximation. Concretely this means
that we exhibit an unbounded complex which is not relatively quasi-isomorphic to the limit of the (relative) injective
resolutions of its truncations. Our methods rely on local cohomology computations, see \cite{MR2355715}.
The failure of being a model approximation is nevertheless rather well behaved, as we never
lose any information about the original complex:

\medskip

\noindent
{\bf Proposition~\ref{prop defaultissplit}}
\emph{Let  $f:\text{\rm tow}(X)\ra Y_{\bullet}$ be a weak equivalence in $\text{\rm Tow}(\A,\I)$
and $g:X\ra \text{\rm lim}(Y_{\bullet})$ be its adjoint. Then,   for any $W\in \I$,
$\A(g,W)$ induces a split epimorphism on homology.}

\medskip

The failure of the standard Quillen pair to be a model approximation is closely related to the ``non-left completeness'' 
of the derived category of some abelian categories, observed by Neeman~\cite{MR2875857}.
To solve this difficulty we introduce in Section~\ref{sec ab4*rel} a relative version of Roos axiom AB4*-$n$, \cite{MR2197371}.

\medskip

\noindent
{\bf Theorem~\ref{thm ab4igivesapprox}}
\emph{Let $\I$ be an injective class and assume that the abelian category $\mathcal{A}$ satisfies axiom {\rm{AB4*}}-$\mathcal{I}$-$n$.
Then the standard Quillen pair
\[
\text{\rm tow}:\text{\rm Ch}(\A)\rightleftarrows \text{\rm Tow}(\A,\I):\text{\rm lim}
\]
is a model approximation.
}
\medskip

For many classes of injective modules this axiom is satisfied. We construct some for Nagata's ring in Subsection~\ref{subsec:noproblem},
the finite Krull dimension case can be understood from this point of view, and Spaltenstein's classical construction also works
for this reason, see Corollary~\ref{cor:Spaltenstein}.

\medskip

\noindent {\bf Acknowledgments.} We would like to thank Michel van
den Bergh for pointing out the relevance of axiom AB4* at a time
when three authors were still thinking that towers approximate unbounded chain
complexes in any relative setting. The fourth author would like to
thank the Mathematics departments at the Universitat Aut\'onoma de
Barcelona and the Australian
National University for providing terrific conditions for a sabbatical.

\section{Chain complexes and  relative weak equivalences}
\label{sec framework}
In this section we recall briefly the definition of an abelian category, introduce the notion of an injective
class, and study the relative weak equivalences that arise in the category of chain complexes in an abelian
category once an injective class has been chosen.

\subsection{Abelian categories}
\label{subsec abelian}
Throughout the paper we work with an \emph{abelian category} $\A$, for
example the category of left modules over
a ring. By an abelian catgory we mean a category with the following
structure~\cite{tohoku}:
\begin{itemize}
\item[(AB0)] {\bf Additivity.} The category $\A$ is additive: finite products and coproducts exist; 
there is a zero object (an object which is both initial and terminal); given two objects $X,Y \in \A$, 
the morphism set $\mathcal{A}(X,Y)$  has an abelian group structure with the zero given by the 
unique morphism that factors through the zero object; the composition of maps  is a bilinear operation.

\item[(AB1)] {\bf Kernels and cokernels.} Any morphism has a kernel and cokernel as defined in~\cite{MR1344215}.

\item[(AB2)] Every monomorphism is the kernel of its cokernel and every epimorphism is the cokernel of its kernel.

\item[(AB3)] {\bf Limits and colimits.} Arbitrary  limits and colimits exist in $\mathcal{A}$.
\end{itemize}

At first we do not ask for any
further properties of products beyond their existence, although
later on we will make a crucial assumption.
Grothendieck's axiom, which we will use, is:

\begin{itemize}
\item[(AB4*)] A countable product of epimorphisms in $\A$ is an epimorphism.
\end{itemize}

%

Let $R$ be a possibly non-commutative unitary ring.  The category of  left $R$-modules, which we call simply $R$-modules  
and denote $R$-Mod, is an abelian category that satisfies axiom AB4*. However, if $X$ is a topological space then the category of 
sheaves of abelian groups on $X$, which is also an abelian category, does not satisfy AB4* in general~\cite[Proposition~3.1.1]{tohoku}.

\subsection{Injective classes}
\label{subsec injective}
Given   an abelian category $\A$ we are interested in understanding relative analogues of monomorphisms  and
injective objects in~$\A$.

\begin{definition} \label{def W-monoepi}
Let $\I$ be a collection of objects in $\mathcal{A}$. A
morphism $f: M \rightarrow N$ in $\A$ is said to be an \emph{$\I$-monomorphism} if,
for any $W \in \I$,  $f^\ast: \mathcal{A}(N,W) \rightarrow
\mathcal{A}(M,W)$ is a surjection of sets.  We say that $\A$ has enough $\I$-injectives if, for any
object $M$, there is an $\I$-monomorphism  $M\ra W$ with $W\in \I$.
\end{definition}

\begin{remark} \label{rem mono properties}
It is clear that a composite of $\I$-monomorphisms  is also an
$\I$-mono\-morphism. 
We say that a morphism $f$ has a retraction if there exists a morphism $r$ such that $rf=\text{id}$.
Any morphism that has a retraction is an $\I$-monomorphism
for any collection~$\I$. Observe also that $\I$-monomorphisms are
preserved under base change: if $f: M \rightarrow N$ is an $\I$-monomorphism, 
then so is its push-out along any morphism $M \rightarrow M'$, by the universal property of a push-out. 
Similarly an  arbitrary coproduct of $\I$-monomorphisms is
an $\I$-monomorphism. 
In general however limits and products of $\I$-monomorphisms may fail to be $\I$-monomorphisms. 
\end{remark}

Given a class of objects $\mathcal{I}$ denote by $\overline{\I}$ the class of retracts  of arbitrary products  of elements of $\I$. 
Since a morphism is an $\I$-monomorphism if and only if it is an $\overline{\I}$-monomorphism, 
without loss of generality
we may assume
that $\I$ is closed under retracts and products so that  $\I = \overline{\I}$.

\begin{definition}\label{def injclass}
A collection of objects $\I$ in  $\A$ is called an \emph{injective class} if $\I$ is closed under retracts and products and 
if $\A$ has  enough $\I$-injectives. 
\end{definition}

It should be pointed out that general products have considerably more retracts  than direct sums. 

\begin{example}\label{ex basicexinj}
The largest injective class $\I$ in $\A$ consists of all the objects in $\A$. Here  $\I$-monomorphisms are morphisms $f:M\ra N$ that have 
retractions.
It is clear that there are enough $\I$-injectives since for any object $N$ the identity $Id_N : N \ra N$ is an $\I$-monomorphism. 

Recall that an object $W$ in an abelian category $\A$ is called injective if, for any monomorphism $f$, $\A(f,W)$ is an epimorphism. 
Assume that any object of $\A$ admits a monomorphism into an injective object, which is the case for example in the category
of left $R$-modules. Then the collection
$\I$ of injective objects in $\A$ is an injective class and $\I$-monomorphisms are the ordinary monomorphisms.
The same holds for the category of $\O_X$-modules for a scheme $X$:
any $\O_X$-module is a submodule of an injective  $\O_X$-module. 
\end{example}

Adjoint functors allow us to construct new injective classes
out of old ones, an idea that goes back to 
Eilenberg-Moore~\cite[Theorem~2.1]{MR0178036}.

\begin{proposition}\label{prop adjfunctor}
Let $l : \B \leftrightarrows \A : r$ be a pair of functors between abelian categories such that $l$ is left adjoint to $r$. 
Let $\I$ be a collection of objects in $\A$.
\begin{enumerate}
\item
A morphism $f$ in $\B$ is an $r(\I)$-monomorphism if and only if $lf$ is  an $\I$-monomorphism in $\A$. 
\item If $lM
\ra W$ is an $\I$-monomorphism in $\A$, then
its adjoint $M\ra rW$ is an $r(\I)$-monomorphism in $\B$.
\item 
If there are enough $\I$-injectives in  $\A$, then there are enough $r(\I)$-injectives in $\B$.
\item If $\I$ is an injective class in $\A$, then the collection of retracts of objects of the form $r(W)$, for $W\in \I$, is an injective class in $\B$.
\end{enumerate}
\end{proposition}


\begin{example}\label{ex tensor}
{\bf Modules of tensor products.} 
Assume now that  $S$ is a commutative ring and $S\ra R$ is a ring homorphism whose image lies in the center of $R$,
hence turns $R$ into  an $S$-algebra. The forgetful functor
$R\text{\rm-Mod}\rightarrow S\text{\rm-Mod}$ is right adjoint to 
$R\otimes_{S}-:S\text{\rm-Mod}\ra  R\text{\rm-Mod}$. By
Example~\ref{ex basicexinj} and
Proposition~\ref{prop adjfunctor}, both the collection of 
$S$-linear summands  of $R$-modules
and the collection of  $S$-linear summands of all injective $R$-modules form 
injective classes of $S$-modules.  A monomorphism relative to the first collection is a homomorphism $f$  
for which $f\otimes_{S}R$ is a split monomorphism.
A monomorphism relative to the second collection is an homomorphism $f$ for which $f\otimes_{S}R$ is a monomorphism.
\end{example}

\begin{example}\label{ex scheme}
{\bf Schemes.} 
Let $f:X\ra Y$ be a morphisms of schemes. The functor $f^{\ast}:\O_Y\text{-Mod}\ra
\O_X\text{-Mod}$ is left adjoint to $f_{\ast}:O_X\text{-Mod}\ra
\O_Y\text{-Mod}$.  It follows that the two collections:  $\O_Y$-modules which are retracts of  $\O_Y$-modules of the form
$f_{\ast}(N)$, for any $\O_X$-module~$N$, and retracts of $\O_Y$-modules of the same form,  but for all injective 
$\O_X$-module~$N$, are injective classes in $\O_Y\text{-Mod}$.
\end{example}

We wish to see to what extent objects in $\mathcal{I}$
behave like usual injective objects, that is
when it is possible to do
homological algebra relative to the class $\mathcal{I}$. We therefore turn to the category $\ch(\A)$ of chain complexes over $\A$ and
to its homotopy category $\K(\A)$.  

\subsection{Relative weak equivalences in $\ch(\A)$}
\label{subsec relativewe}
In this work we mostly consider \emph{homological} complexes  (i.e. differentials lower degree by one) in $\A$:
$X =(\cdots\ra X_{i}\xrightarrow{d_i} X_{i-1}\ra\cdots)$.
The category of such chain complexes in $\mathcal{A}$ is denoted by $\text{Ch}(\mathcal{A})$. 
We identify $\A$ with the full subcategory of $\ch(\A)$ of those complexes concentrated in degree $0$
and will use the topologist's suspension symbol $\Sigma X$ for the shifted complex sometimes denoted by~$X[1]$.

The only examples of  \emph{cohomological} complexes that we consider are
complexes of abelian groups of the form  $\A(X,W)$ for some $X\in \text{Ch}(\mathcal{A})$ and $W\in \A$. 
As usual, if $X_k$ is in homological degree $k \in \Z$, we put $\A(X_k,W)$ in cohomological degree $-k$. 
The key definition for doing relative homological algebra is the following.

\begin{definition}\label{def Imono}
Let $k \in \Z$ be an integer. A morphism $f:X\ra Y$ in $\ch(\A)$ is called  a \emph{$k$-$\I$-weak equivalence} if and only if,
for any $W\in \I$, the induced morphism of cochain complexes $\A(f,W):\A(Y,W)\ra \A(X,W)$ induces an isomorphism in cohomology in 
degrees $n\geq -k$ and a monomorphism in degree $-k-1$. A morphism that is a $k$-$\I$-weak equivalence for all $k \in \Z$ is called an 
\emph{$\I$-weak equivalence}. 
\end{definition}

\begin{definition}\label{def Itrivial}
An object $X$ in $\ch(\A)$  is called {\em $\I$-trivial} when $X\ra 0$ is an $\I$-weak equivalence, i.e. when
$\A(X,W)$ is an acyclic complex of  abelian groups for all $W\in \I$.  It is called
\emph{$k$-$\I$-connected} if $X\ra 0$ is an $k$-$\I$-weak equivalence, i.e., when
 $\A(X,W)$ has trivial cohomology in degrees $n\geq -k$ for all $W\in\I$.
\end{definition}

Let us see what these definitions mean for the examples we introduced in the previous subsection.

\begin{example}\label{exs relweq}
We study first the case when $\I$ is the injective class of  all objects of $\A$.
For  an object $ M \in \A$ and an integer  $k$ denote by $D_k(M)$  the ``disc'' chain complex
\[
\xymatrix{
\cdots  0 \ar[r] & M \ar@{=}^{Id_M}[r] & M \ar[r] & 0 \ar[r] & \cdots
}
\]
where the two copies of $M$ are in homological degrees $k$ and $k-1$ respectively.
Complexes of the form $D_{k}(M)$  are prototypical examples of contractible complexes. 

A morphism of chain complexes  $f:X\ra Y$ is an $\I$-weak equivalence if and only if it is a homotopy equivalence. 
A chain complex is $\I$-trivial if and only if it is isomorphic to  $\bigoplus_{i} D_{k_i}(M_i)$
for some sequence of objects $M_i\in \A$ and integers $k_i \in \mathbf Z$.
\end{example}

\begin{example}\label{exs injectiverel}
Let us assume that classical injective objects form an injective class, i.e. any object in $\A$ is a subobject of an injective object. 
As the functors $\A(-,W)$ are exact when $W$ is injective,  a morphism of complexes $f:X\ra Y$  in $\ch(\A)$ is an $\I$-weak equivalence 
if and only if it is a quasi-isomorphism. A chain complex is $\I$-trivial if and only if it has trivial homology.
\end{example}

\begin{example}\label{exs adjointrel}
Consider a pair of adjoint functors $l : \B \leftrightarrows \A : r$ between abelian categories and $\I$ an injective class in $\A$. 
According to Proposition~\ref{prop adjfunctor}.(4), the collection $\mathcal{J}$ of retracts of objects of the form $r(W)$, 
for $W\in \I$, forms an injective class in $\B$. 
By applying $l$ and $r$ degree-wise, we get an induced pair of adjoint functors, denoted by the same symbols:
$l : \ch(\B) \leftrightarrows \ch(\A) : r$.
A morphism $f:X\ra Y$ in $\ch(\B)$ is a $\mathcal{J}$ -weak equivalence if and only if
$l(f):l(X)\ra l(Y)$ is an $\I$-weak equivalence in $\ch(\A)$.
\end{example}

Our next example is based on the classification of injective classes of injective objects in a module category given in \cite{CPSinj}, 
to which we refer for more details. Let us recall however that given an ideal $I$ in $R$ and an element $r$ outside of $I$, then
$(I\colon r)$ denotes the ideal $\{ s \in R \, | \, sr \in I\}$. This example will play an important role in the final sections of this article.

\begin{example}\label{ex injcall-stideal}
Let $R$ be a commutative ring and $\L$ be a saturated set of ideals in~$R$. This means that $\L$ is a set of proper ideals of $R$ closed 
under intersection and the construction $(I\colon r)$; moreover if an ideal $J$ has the property that $(J\colon r)$ is contained in some
ideal in $\L$ for any element $r \notin J$, then $J$ itself must belong to $\L$.

Consider the injective class $\E(\L)$ that  consists of retracts of products of injective envelopes $E(R/I)$ for $I\in\L$.
A morphism $f:X\ra Y$ in $\ch(R)$ is an $\E(\L)$-weak equivalence if and only if $\text{Hom}(\text{H}_n(f),E(R/I))$ is a bijection for any $n$ and 
$I\in\L$. This happens if and only if the  annihilator of  any element in either $\text{ker}(\text{H}_n(f))$ or 
$\text{coker}(\text{H}_n(f))$ is not included in any ideal that belongs to $\L$.
\end{example}

We  denote the class of $\I$-weak equivalences by $\W_\I$ or simply $\W$ if there is no ambiguity for the choice of the ambient injective 
class~$\I$. Isomorphisms are always  $\I$-weak equivalences and $\I$-weak equivalences satisfy the ``2 out of 3'' property, as the stronger
``2 out of 6'' property from \cite[Definition~4.5]{MR2102294} holds.

\begin{lemma}\label{lem 2outof6}
The class $\mathcal{W}_\I$ of $\I$-weak equivalences satisfies the $2$ out of $6$ property: given any three composable maps 
\[
\xymatrix{ X \ar[r]^u & Y \ar[r]^v & Z \ar[r]^w & T}
\]
if $vu$ and $wv$ are in $\mathcal{W}$ then so are $u,v,w$ and $wvu$.
\end{lemma}

\begin{proof}
Fix an object $W \in \I$. Then $\A(vu,W ) = \A(u,W) \circ \A(v,W)$ is a quasi-isomorphism, hence $\A(v,W)$ induces an 
epimorphism in cohomology. Similarly, from the fact that $\A(wv,W)$ is a quasi-isomorphism we get that $\A(v,W)$ induces 
a monomorphism in cohomology, hence $v$ belongs to~$\mathcal{W}$. Since quasi-isomorphisms satisfy the $2$ out of $3$ property 
we get that $\A(u,W)$ and $\A(w,W)$ are quasi-isomorphisms and $v,w$ is in $\mathcal{W}$. By closure under composition so is~$wvu$. 
\end{proof}

Here are some elementary properties of $\I$-weak equivalences:

\begin{proposition}\label{prop propunch}
Let $I$ be an injective class in an abelian category $\A$.
\begin{enumerate}
\item A chain homotopy equivalence in $\ch(\A)$  is an $\I$-weak equivalence.
\item  A morphism $f:X\ra Y$ in $\ch(\A)$ is an $\I$-weak equivalence if and only if the cone
$\text{\rm Cone}(f)$ is $\I$-trivial.
\item Coproducts of $\I$-weak equivalences are $\I$-weak equivalences.
\item A contractible chain complex in $\ch(\A)$ is $\I$-trivial.
\item  Coproducts of $\I$-trivial complexes are $\I$-trivial.
\item A complex $X$ is $k$-$\I$-connected  if and only if, for any $i\leq k$, $d_i:\text{\rm coker}(d_{i+1})\ra X_{i-1}$ is an $\I$-monomorphism.
\item A complex $X$ is $\I$-trivial if and only if, for any $i$,  $d_i:\text{\rm coker}(d_{i+1})\ra X_{i-1}$ is an $\I$-monomorphism.
\item Let $X$ be a complex such that, for all $i$, $\text{\rm coker}(d_{i+1})\in\I$. Then  $X$ is $\I$-trivial if and only if
$X$ is isomorphic to $\bigoplus D^{i}(W_i)$.
\end{enumerate}
\end{proposition}

\begin{proof}
Point (1) is a consequence of the fact that $\A(-,W)$ is an additive functor.

(2)  The cone of $\A(f,W):\A(Y,W)\ra \A(X,W)$ is isomorphic to the shift of the
complex $\A(\text{\rm Cone}(f),W)$, for any $W\in\A$. Thus $\A(f,W)$ is a quasi-isomorphism if and only if $\A(\text{\rm Cone}(f),W)$ is acyclic. 

Point (3) is a consequence of two facts. First,  $\A(-,W)$ takes coproducts in $\A$ into products of abelian groups. 
Second, products of quasi-isomorphisms of chain complexes of abelian groups are quasi-isomorphisms.

Point (4) is a special instance of Point (1), and given (4), Point (5) is a special case  of Point (3).

(6) The kernel of $\A(d_{i+1},W)$ is $\A(\text{coker}(d_{i+1}),W)$. 
Thus the $i$-th cohomology of  $\A(X,W)$ is trivial if and only if the morphism $\A(X_{i-1},W)\ra \A(\text{coker}(d_{i+1}),W)$ 
induced by $d_i$ is an epimorphism. By definition this happens if and only if the morphism $d_i:\text{coker}(d_{i+1})\ra X_{i-1}$ is an $\I$-monomorphism.

(7) This is a consequence of (6).

(8) If $X$ can be expressed as a direct sum $\bigoplus D_i(W_i)$, then $X$ is contractible and according to (4) it is $\I$-trivial. 
Assume now that $X$ is $\I$-trivial.  Define $W_i:=\text{coker}(d_{i+1})$. According to (6), the morphism 
$d_i:\text{coker}(d_{i+1})\ra X_{i-1}$ is an $I$-monomorphism. As $\text{coker}(d_{i+1})$ is assumed to belong to $\I$, 
it follows that the morphism $d_i:\text{coker}(d_{i+1})\ra X_{i-1}$   has a retraction.  This retraction can be used to define  
a morphism of chain complexes $X\ra D_{i}(W_i)$.  By assembling these morphisms together we get the desired isomorphism 
$X\ra \bigoplus  D_{i}(W_i)$.
\end{proof}

\section{The relative derived category as a large category}
Doing homological algebra relative to an injective class $\mathcal{I}$ amounts to inverting  the morphisms in $\W$ to form the 
relative derived category $D(\A;\I) = \ch(\A)[\W_{\I}^{-1}]$.
The formalities of inverting a class of morphisms in a category 
are well understood.
But there is a problem that, without some extra structure,
the resulting category turns 
out to be a large category in general, i.e. with classes of morphisms between two objects instead of sets of morphisms. This becomes an 
issue if one wants to further localize in this category or study its quotients. Let us nevertheless put this set-theoretical issue 
aside for the moment, and remind the reader
of the classical construction of the relative derived category $D(\A;\I)$.
In particular we recall that the classical results endow $D(\A;\I)$ with a
canonical triangulated structure.

As chain homotopy equivalences are in particular $\I$-equivalences the localization functor $\ch(\A) \rightarrow \ch(\A)[\W_{\I}^{-1}]$, if it exists, 
factors through the canonical localization functor $\ch(\A) \rightarrow \K(\A)$, where $\K(\A)$ is the homotopy category of chain complexes. 
The category  $\K(\A)$ is a triangulated category, and we exploit this fact and the theory of \emph{null systems}, 
\cite[Section 10.2]{MR2182076}, to construct the relative derived category.  

\begin{definition}\label{def nullsystem}
Let $\mathcal{T}$ be a triangulated category and $\mathcal{N}$ be a class of objects in $\mathcal{T}$ closed under isomorphisms. 
Then $ \mathcal{N}$ is a \emph{null system} if and only if the following axioms are satisfied:
\begin{enumerate}
 \item[(N0)] The zero object of $\mathcal{T}$ is in $\mathcal{N}$.
 \item[(N1)] For any $X \in \mathcal{T}$, $X \in \mathcal{N} \Leftrightarrow \Sigma X \in \mathcal{N}$.
 \item[(N2)] Given a triangle $\xymatrix{  X \ar[r]^u& Y \ar[r]^{v} &  Z  \ar[r]^w & \Sigma X}$ in $\mathcal{T}$, if $X,Z \in \mathcal{N}$ 
 then $Y \in \mathcal{N}$.
\end{enumerate}
\end{definition}

The main property of null systems is that it allows us to
construct the Verdier quotient $\mathcal{T}/\mathcal{N}$ by
a simple calculus of fractions 
(although recall that this quotient may have proper classes of morphisms). 
For a proof of the following proposition we refer the reader to \cite{MR2182076}.

\begin{proposition}\label{prop mainusenullsyst}
Given a triangulated category $\mathcal{T}$ and a null system $\mathcal{N}$ in $ \mathcal{T}$, set:
\[
\mathcal{S}(\mathcal{N}) = \{ f:X \rightarrow Y \in \mathcal{T} \ \vert \ \exists \ a \ triangle \ \xymatrix{ X \ar[r]^f & Y \ar[r]  &  Z  \ar[r] & \Sigma X} \textrm{ with } Z \in \mathcal{N} \}
\]
Then $\mathcal{S}(\mathcal{N})$ admits a left and right calculus of fractions. In particular:
\begin{enumerate}
\item The localization $\mathcal{T}/\mathcal{N} := \mathcal{T}[\mathcal{S}(\mathcal{N})^{-1}]$ exists.
\item Let us declare the
  isomorphs in $\mathcal{T}/\mathcal{N}$ of images of triangles in 
  $\mathcal{T}$, via the canonical quotient  functor $\mathcal{T} \rightarrow \mathcal{T}/\mathcal{N}$, to be the triangles in $\mathcal{T}/\mathcal{N}$.
  Then the category 
$\mathcal{T}/\mathcal{N}$ becomes triangulated and the canonical quotient functor is triangulated.
\end{enumerate}

\end{proposition}

We apply this to our situation of interest, where we want to invert the relative equivalences, i.e. kill the cones of
$\mathcal{W}_\I$-equivalences, which are $\I$-trivial by Proposition~\ref{prop propunch}.(2).

\begin{proposition}\label{prop Wnullsistem}
In $\mathcal{K}(\A)$, the homotopy category of $\mathcal{A}$ with its standard triangulated structure, the class $\mathcal{WN}$ of 
$\I$-trivial objects
forms a null system.
\end{proposition}

\begin{proof} 
Axioms (N0) and (N1) hold by definition of $\I$-triviality, see Definition~\ref{def Itrivial}.
%
%

(N2) Let $W$ be an object
in $\I\subset\A$.
Let $\xymatrix{ X \ar[r] & Y \ar[r] & Z \ar[r] & \Sigma X}$ be a triangle in $\mathcal{K}(\A)$, with $X,Z \in \mathcal{WN}$. 
Applying the 
functor $\A(-,W)$ to the triangle
we deduce a triangle, in the homotopy category $\mathcal{K}(\mathrm{Ab})$
where $\mathrm{Ab}$ is the abelian category of abelian groups,
\[
\xymatrix{ \A(\Sigma X,W) \ar[r] & \A(Z,W) \ar[r] & \A(Y,W) \ar[r] & \A(X,W) 
}
\]
Since $\A(Z,W)$ and $\A(X,W)$ are both acyclic so
is $\A(Y,W)$.
\end{proof}

From the general theory it follows that the class $\mathcal{W}_{\mathcal{I}}$ of $\I$-equivalences admits simple right and left calculuses of fractions. As a consequence we have:

\begin{corollary}\label{cor D(A,I)exists}
Let $\A$ be an abelian category, $\I$ a class of injective objects, and $\mathcal{W}_\I$ the associated class of $\I$-weak equivalences.
\begin{enumerate}
\item The localization $\ch(\A)[\mathcal{W}_\I^{-1}]=\colon D(\A;\I)$ exists and  has a natural  triangulated category structure 
which is  functorial with respect to inclusions of classes of relative weak equivalences.
\item The canonical functor $\mathcal{K}(\A) \rightarrow D(\A;\I)$ is triangulated.
\item The class $\mathcal{W}_\I$ is saturated: a map $f \in \mathcal{K}(\A)$ is an isomorphism in $D(\A;\I)$ if and only if $f \in \mathcal{W}_\I$.
\end{enumerate}

\end{corollary}

\begin{proof}
The only non-immediate consequence from Proposition~\ref{prop mainusenullsyst} is point (3), which is a consequence of the ``2 out of 6'' property,
Lemma~\ref{lem 2outof6}, see \cite[Prop. 7.1.20]{MR2182076}.
\end{proof}
%
%

\section{Model categories and model approximations}
\label{sec modelcat}
We now present our set-up for doing homotopical algebra. In homotopy theory a convenient framework for localizing categories 
and constructing derived functors is given by Quillen model categories; we use the term model category as defined in~\cite{MR1361887}. 
There are however situations in which, either it is very hard to construct a model structure, or one simply does not know whether such a 
structure does exist.  We will explain  how to localize and  construct right derived functors in a more  general context than model categories. 
We do not to try to impose a model structure on a given category with weak equivalences directly but rather use model categories to 
\emph{approximate} the given category. 

Let $\C$ be a category and $\W$ be a collection of morphisms in $\C$ which contains
all isomorphisms and satisfies the``2 out of 3'' property: if $f$ and $g$ are composable morphism in $\C$ 
and 2 out of $\{f,g,gf\}$ belong to $\W$ then so does the third. We call elements
of $\W$  weak equivalences and a pair $(\C,\W)$ a category with weak equivalences.  
The following definitions come from \cite[3.12]{MR1879153}.

\begin{definition}\label{def quillen}
A  \emph{right Quillen pair} for $(\C,\W)$ is a model category $\M$ and a pair of functors 
 $l:\C \rightleftarrows \M:r$ satisfying the
following conditions:
\begin{enumerate}
\item $l$ is left adjoint to $r$;
\item if $f$ is a weak equivalence in $\C$, then  $lf$ is a weak equivalence in $\M$;
\item if $f$ is a weak equivalence between fibrant objects in $\M$, then
$rf$ is a weak equivalence in $\C$.
\end{enumerate}
\end{definition}

\begin{definition}\label{def approx}
We say that an object $A$ in $\C$ is \emph{approximated} by a right Quillen pair $l:\C \rightleftarrows \M:r$ if the 
following condition is satisfied:
\begin{itemize}
\item[(4)] if $lA\ra X$ is a weak equivalence in $\M$ and $X$ is fibrant,
then its adjoint $A\ra rX$ is a weak equivalence in $\C$.
\end{itemize}

If all objects of $\C$ are approximated by $l:\C \rightleftarrows \M:r$,  then this Quillen pair is called a 
\emph{right model approximation} of $\C$.
\end{definition}

For an object $A$ to be approximated by a Quillen pair, we only need the existence of \emph{some} fibrant object $X$ in the model category
together with a weak equivalence $lA\ra X$ and such that the adjoint map is a weak equivalence. Condition (4) is then automatically satisfied 
for \emph{any} such fibrant object.

Let us fix a right Quillen pair  $l:\C \rightleftarrows \M:r$ and  choose a full subcategory $\D$ of $\C$ with the following properties: 
all objects in $\D$ are approximated by the Quillen pair and, for a weak equivalence $f:X\ra Y$, if one of $X$ and $Y$ belongs to 
$\D$ then so does the other ($\D$ is closed under weak equivalences). We are going to think of $\D$ as a category with weak 
equivalences given by the morphisms in $\D$ that belong to $\W$.
Here are some fundamental properties of  this category, whose proofs extend those for model approximations in~\cite[Section 5]{MR1879153}:

\begin{proposition}\label{prop quillenpair}
\begin{enumerate}
\item A morphism $f$ in $\D$ is a weak equivalence if and only if $lf$ is a weak equivalence in $\M$.
\item The localization $\text{\rm Ho}(\D)$  of $\D$ with respect to weak equivalences exists and can be constructed as follows:
objects of $\text{\rm Ho}(\D)$ are the same as objects of $\D$ and
$\text{\rm mor}_{\text{\rm Ho}(\D)}(X,Y)=\text{\rm mor}_{\text{\rm Ho}(\M)}(lX,lY)$.
\item A morphism in $\D$ is a weak equivalence if and only if it induces an isomorphism in  $\text{\rm Ho}(\D)$.
\item The class of weak equivalences in $\D$ is closed under retracts.
\item Let $F:\C\ra \T$ be a functor. Assume that the composition $Fr:\M\ra \T$ takes weak equivalences between 
fibrant objects in $\M$ to isomorphisms in $\T$. Then the right derived functor of the restriction
$F:D\ra \T$ exists and is given by $A\mapsto F(rX)$, where $X$ is a fibrant replacement of $lA$ in $\M$. 
\end{enumerate}
\end{proposition}

\begin{proof}
(1)  Assume that $lf:lA\ra lB$ is a weak equivalence in $\M$. 
Choose a weak equivalence  $lB\ra Y$  with fibrant target $Y$. By taking adjoints we form the
following commutative diagram in $\D$:
\[
\xymatrix{
A\ar[r]^{f}\ar[dr] &B\ar[d]\\
& rY
}
\]
Since  $A$ and $B$ belong to $\D$,
the morphisms $A\ra rY$ and $B\ra rY$ are weak equivalences, as their
adjoints are so. By the ``two out of three'' property,  $f$ is then also a weak equivalence.

(2) Let $\alpha:\D\ra \T$ be a functor that sends weak equivalences to isomorphisms. We prove that there is a unique functor 
$\beta:\text{\rm Ho}(\D)\ra \T$ for which the composition
$\D\ra \text{\rm Ho}(D)\xrightarrow{\beta}\T$ equals $\alpha$.
On objects we have no choice, we define $\beta(A):=\alpha(A)$.

Let $A$ and $B$ be objects in $\D$. Since $\text{\rm mor}_{\text{\rm Ho}(\D)}(A,B)=\text{\rm mor}_{\text{\rm Ho}(\M)}(lA,lB)$, 
a morphism  $[f]:A\ra B$ in $\text{\rm Ho}(\D)$ is given  by a sequence of morphisms in $\M$:
\[
lA\xrightarrow{a_1} A_1\xleftarrow{a_2} A_2\xrightarrow{g} B_1\xleftarrow{b} lB
\]
where $a_1$ is a weak equivalence with fibrant target $A_1$, $a_2$ is a weak equivalence with fibrant and cofibrant 
domain $A_2$, and $b$ is a weak equivalence with fibrant target~$B_1$. By adjunction we get a sequence of morphisms in $\D$:
\[
A\xrightarrow{\overline{a}_1} rA_1\xleftarrow{ra_2}rA_2\xrightarrow{rg} rB_1\xleftarrow{\overline{b}} B
\]
Note that $\overline{a}_1$, $ra_2$, and $\overline{b}$ are weak equivalences. 
We define $\beta([f])$ to be the unique morphism in $\T$ for which the following diagram commutes:
\[
\xymatrix{
\alpha(A)\ar[rr]^{\beta([f])} \ar[d]_{\alpha(\overline{a_1})}\ar[dr]& & \alpha(B)\ar[d]^{\alpha(\overline{b})}\\
\alpha(rA_1) & \alpha(rA_{2})\ar[l]^{\alpha(ra_2)} \ar[r]_{\alpha(rg)}\ar[ur]&\alpha(rB_1)
}\]
Since $\alpha$ takes weak equivalences to isomorphisms such a morphism
$\beta([f])$ exists and is unique. One can finally check that this process defines the desired functor $\beta:\text{Ho}(\D)\ra \T$.

(3) is a consequence of (1) and (2). Point (4) follows from (3).

(5) For any object $A\in \D$ let us fix a fibrant replacement $lA\ra RA$ in $\M$. For any morphism
$f:A\ra B$ in $\D$ let us fix a morphism $Rf:RA\ra RB$ in $\M$ for which the following diagram commutes:
\[
\xymatrix{
lA\ar[r]^{lf}\ar[d] & lB\ar[d]\\
RA\ar[r]^{Rf} & RB
}
\]
Since $Fr$ takes weak equivalences between fibrant objects to isomorphisms,
the association $A\mapsto F(rRA)$ and $f\mapsto F(rRf)$ defines a functor $RF:\D\ra \T$.
We claim that $RF$ together with the natural transformation given by $F(A\ra rRA)$
is the right derived functor of $F:\D\ra \T$.
It is clear that $RF$ takes weak equivalences to isomorphisms. Let $G:\D\ra \T$ be a functor that takes weak equivalences to 
isomorphisms and let $\mu:F\ra G$ be a natural transformation. For any $A\in\D$ define $F(rRA)\ra G(A)$ to be the unique 
morphism that fits into the following commutative diagram in $\T$:
\[
\xymatrix{
F(A)\ar[r]^{\mu_{A}}\ar[d] & G(A)\ar[d]\\
F(rRA)\ar[r]^{\mu_{rRA}}\ar[ur] & G(rRA)
}
\] 
Such a morphism does exist since $G(A)\ra G(rRA)$ is an isomorphism as $A\ra rRA$ is a weak equivalence.
\end{proof}

\section{Towers}\label{sec tower}
For a given category with weak equivalences $(\C,\W)$ and a full subcategory $\D$ our strategy is to construct 
a right Quillen pair $l:\C\rightleftarrows \M:r$ which approximates objects of $\D$.
We can then use this Quillen pair to localize $\D$ with respect to weak equivalences and
construct right derived functors as explained in Proposition~\ref{prop quillenpair}. For this strategy to work we need 
adequate examples of model categories.  The purpose of this section is to show how to assemble model categories 
together to build new model categories that are suitable to approximate $\D$. Such diagrams of model categories 
have appeared meanwhile in work of Greenlees and Shipley, \cite{MR3254575}, see also Bergner's construction
of a homotopy limit model category for a diagram of left Quillen functors, \cite{MR2914609}.
We include the following definitions and results to fix notation and so as to be able to refer to specific
constructions in the next sections.

We start with a tower $\T$ of model categories consisting of a sequence of model categories
$\{\T_n\}_{n\geq 0}$ and a sequence of Quillen functors $\{l:\T_{n+1}\rightleftarrows \T_n:r\}_{n\geq 0}$:
for any $n$, $l$ is left adjoint to $r$ and $r$ preserves fibrations and acyclic fibrations. 
The model categories in a tower $\T$ can be assembled to form its category of towers.

\begin{definition}
The objects $a_{\bullet}$ of the \emph{category of towers} $\text{Tow}(\T)$ are sequences $\{a_n\}_{n\geq 0}$ of objects $a_n\in \T_n$ 
together with a sequence of \emph{structure morphisms} $\{a_{n+1}\ra r(a_n)\}_{n\geq 0}$.  
The set of morphisms in $\text{Tow}(\T)$ between $a_{\bullet}$ and $b_{\bullet}$ consists of sequences of morphisms 
$\{f_n:a_n\ra b_n\}_{n\geq 0}$ for which the following squares commute:
\[
\xymatrix{
a_{n+1}\ar[d]_{f_{n+1}}
\ar[r] &r(a_{n})\ar[d]^{r(f_{n})}\\
b_{n+1}\ar[r]  & r(b_{n})
}\]
We write $f_{\bullet}:a_{\bullet}\ra b_{\bullet}$ to denote the morphism $\{f_n:a_n\ra b_n\}_{n\geq 0}$ in 
$\text{Tow}(\T)$. 
\end{definition}

The following construction will be useful to describe a model structure on $\text{Tow}(\T)$.
For a morphism   $f_{\bullet}:a_{\bullet}\ra b_{\bullet}$, define $p_0:=b_0$ and, for $n>0$, define:
\[
p_n:=\text{lim}\big(b_{n}\rightarrow r(b_{n-1})\xleftarrow{r(f_{n-1})} r(a_{n-1})\big)
\]
Set $\alpha_{0}:a_0\ra p_0$ to be given by $f_0$ and $\beta_0:p_0\ra b_0$ to be the identity. For $n>0$, 
let $\beta_n:p_n\ra b_n$ and $\overline{\alpha}_n:p_n\ra r(a_{n-1})$ be the projection from the inverse limit onto the components $b_n$, 
respectively  $r(a_{n-1})$. Finally $\alpha_{n}:a_{n}\rightarrow p_n$ is the unique morphism
for which the following diagram commutes: 
\[
\xymatrix{
a_{n} \ar@(d,l)[ddr] \ar@(r,u)[rrd] \ar@{->}[dr]^-{\alpha_n} & & \\
& p_n \ar^(.38){\overline{\alpha}_{n}}[r] \ar@{->}[d]^{\beta_n} & r(a_{n-1}) \ar@{->}[d]^{r(f_{i-1})} \\
& b_{n} \ar[r] & r(b_{n-1})
}
\]
The sequence $\{p_n\}_{n\geq 0}$ together with
morphisms $\{p_{n+1}\xrightarrow{\overline{\alpha}_{n+1}} r(a_i)\xrightarrow{r(\alpha_n)} r(p_n)\}_{n\geq k}$ defines an object
$p_{\bullet}$ in $\text{Tow}(\T)$. Moreover $\{\alpha_n:a_n\ra p_n\}_{n\geq 0}$ and $\{\beta_n:p_n\ra b_n\}_{n\geq 0}$ define morphisms
$\alpha_{\bullet}:a_{\bullet}\rightarrow p_{\bullet}$ and $\beta_{\bullet}:p_{\bullet}\ra b_{\bullet}$
whose composite is $f_{\bullet}$.   For example, let  $\ast_{\bullet}$ be given by the sequence consisting 
of the terminal objects $\{\ast\}_{n\geq 0}$ in $\T_n$ and  $f_{\bullet}:a_{\bullet}\ra \ast_{\bullet}$ be the unique morphism
in $\text{Tow}(\T)$. Then $p_0=\ast$, and, for $n>0$, $p_n=r(a_{n-1})$. The morphism $\alpha_n:a_n\ra p_n=r(a_{n-1})$ is given 
by the structure morphism of $a_{\bullet}$.

\begin{definition}
A morphism $\{f_n:a_n\ra b_n\}_{n\geq 0}$ in $\text{Tow}(\T)$ is  a \emph{weak equivalence} (respectively a \emph{cofibration}) 
if, for any $n\geq0$, the morphism $f_n$ is a weak equivalence  (respectively a cofibration) in $\T_n$. It is a \emph{fibration} if
$\alpha_n:a_{n}\rightarrow p_n$ is a fibration in $\T_{n}$ for any $n\geq 0$.
\end{definition}

For example the morphism $a_{\bullet}\ra\ast_{\bullet}$ is a fibration if and only if
$a_0$ is fibrant  in $\T_0$ and the structure morphisms  $a_n\ra r(a_{n-1})$ are  fibrations in $\T_n$ for all $n$.

The following result is a particular case of the existence of the injective model structure for diagrams of model
categories, \cite[Theorem~3.1]{MR3254575}. We provide some details of the proof as we will refer to the explicit 
construction of the factorizations.

\begin{proposition}\label{prop modelontowers}
The above choice of weak equivalences, cofibrations, and fibrations  equips  $\text{\rm Tow}(\T)$ with
a model category structure.
\end{proposition}

\begin{proof}
First, the category $\text{Tow}(\T)$ is bicomplete, as limits and colimits are formed ``degree-wise''. 
The structural morphisms of the limit are the limits of the structural morphisms since the functors $r$, as right adjoints, commute with limits. 
For colimits, one considers the adjoints $l(a_{n+1})\rightarrow a_n$ of the structural morphisms, and takes colimits
$l(\text{colim}(a_{n+1}))\cong \text{colim}\,l(a_{n+1}) \rightarrow \text{colim}(a_{n})$. The adjoint morphisms 
$\text{colim}(a_{n+1})\rightarrow r(\text{colim}(a_{n}))$ are precisely the  structural morphisms of the colimit.

The ``2 out of 3'' property (MC2) for weak equivalences and the fact that retracts of  weak equivalences (respectively cofibrations) are weak
equivalences (respectively cofibrations) follow immediately from the same properties for the categories~$\T_n$.
To prove axiom (MC3), notice that if $\{c_n\rightarrow d_n\}_{n\geq 0}$ is a
retract of a fibration  $\{a_n\rightarrow b_n\}_{n\geq 0}$, then
$c_0\rightarrow d_0$ is a fibration in $\T_0$. Next consider the following commutative diagram for $n>0$:
\[
\xymatrix{
d_{n} \ar[r] \ar[d] & r(d_{n-1}) \ar[d] & r(c_{n-1}) \ar@{->>}[l] \ar[d] & q_n \ar[d] & \ar[l] c_n \ar[d]\\
b_{n} \ar[r] \ar[d] & r(b_{n-1}) \ar[d] & r(a_{n-1}) \ar@{->>}[l] \ar[d]  \ar@{~>}[r]^(.6){\text{lim}} & p_n \ar[d] & \ar@{->>}[l] a_n \ar[d] \\
d_{n}  \ar[r] & r(d_{n-1})  & r(c_{n-1}) \ar@{->>}[l] &   q_n & \ar[l] c_n
}
\]
where the penultimate column has been obtained by taking pull-backs. 
By the retract property in $\mathcal{T}_n$ the morphism $c_n \rightarrow
q_n$ is fibration, for any $n>0$, and therefore so is $\{c_n \rightarrow
d_n\}_{n \geq 0}$ in $\text{Tow}(\T)$.

Let us prove now the right and left lifting properties (MC4).
Consider a commutative diagram:
\[
\xymatrix{
a_\bullet \ar@{^{(}->}[d]^{\sim} \ar[r] & c_\bullet \ar@{->>}[d] \\
b_{\bullet} \ar[r] & d_\bullet
}
\]
where the indicated arrows are respectively an acyclic cofibration
and a fibration. In degree $0$, a lift $b_0 \rightarrow c_0$ is
provided by the model structure on $\mathcal{T}_0$. We construct the lift inductively. Take the solved lifting problem at level $n$ and
complete with the  structural maps to get the following commutative cube:
\[
\xymatrix{
& r(a_n) \ar[rr] \ar[dd]|\hole & & r(c_n) \ar@{->>}[dd] \\
a_{n+1} \ar[ur] \ar[rr] \ar@{^{(}->}[dd]^{\sim} & & c_{n+1}
\ar[ur] \ar[dd] & \\
& r(b_n) \ar[rr]|\hole \ar@/ _1pc/[uurr]|!{[ur];[dr]}\hole & & r(d_n )\\
b_{n+1} \ar[ur] \ar[rr] & & d_{n+1} \ar[ur] &
}
\]
As above, denote by $q_{n+1}$ the pull-back of $d_{n+1} \rightarrow
r(d_n) \leftarrow r(c_n)$. By the universal property of the pull-back there is a morphism $b_{n+1} \rightarrow q_{n+1}$
that makes the resulting diagram commutative. Since by definition
$c_{n+1} \rightarrow q_{n+1}$ is a fibration, the lifting problem
\[
\xymatrix{
a_{n+1} \ar[r] \ar@{^{(}->}[d]^{\sim} & c_{n+1} \ar@{->>}[d]\\
b_{n+1} \ar[r] & q_{n+1}
}
\]
has a solution, which is the desired morphism. The proof for the
right lifting property for acyclic fibrations with respect to cofibrations is analogous.

Finally, to prove the factorization axiom (MC5), consider a morphism
$a_\bullet \rightarrow b_\bullet$. The morphism $a_0 \rightarrow
b_0$ can be factored  as an acyclic cofibration followed by a
fibration (respectively  as a cofibration followed by an acyclic
fibration) because (MC5) holds in $\mathcal{T}_0$. We construct a factorization $a_{n+1} \hookrightarrow
c_{n+1} \twoheadrightarrow b_{n+1}$ by
induction on the degree. Consider the following  commutative
diagram:
\[
\xymatrix{
a_{n+1} \ar[dd] \ar[rr] \ar[dr] & & r(a_n) \ar[d] \\
& z_{n+1} \ar[r] \ar@{->>}[ld] & r(c_n) \ar@{->>}[d] \\
b_{n+1} \ar[rr] & & r(b_n)
}
\]
where the right column is obtained by applying the functor $r$ to
the factorization at level $n$ and the bottom right square is a pull-back.
Since both $r$ 
and cobase-change preserve (acyclic) fibrations, $z_{n+1} \rightarrow
b_{n+1}$ is an (acyclic) fibration as long as $c_n \rightarrow b_n$ is.
It is now enough to factor  $a_{n+1} \rightarrow z_{n+1}$ in
$\mathcal{T}_{n+1}$ in the desired way to obtain the factorization of
$a_{n+1} \rightarrow b_{n+1}$.
\end{proof}

\begin{example}\label{ex tower}
Let $\M$ be a model category. The constant sequence $\{\M\}_{n\geq 0}$ together with the sequence of identity functors
$\{\text{id}:\M\rightleftarrows \M:\text{id}\}_{n\geq 0}$ forms a tower of model categories.
Its category of towers can be identified with the category of functors $\text{Fun}({\mathbf N}^{op}, \M)$, where   $\mathbf N$ is the poset   whose objects
are natural numbers, ${\mathbf N}(n,l)=\emptyset$ if
$n>l$,  and ${\mathbf N}(n,l)$ consists of one element if $n\leq l$.  The model structure on
$\text{Fun}({\mathbf N}^{op}, \M)$, given by Proposition~\ref{prop modelontowers},
coincides with the standard model structure on the functor category $\text{Fun}({\mathbf N}^{op}, \M)$ (see~\cite{MR1879153}). 
For example, a functor $F$ in $\text{Fun}({\mathbf N}^{op}, \M)$ is fibrant if
the object $F(0)$ is fibrant in $\M$ and, for any $n>0$, the morphism  $F(n)\ra F(n-1)$
 is a fibration in $\M$. A morphism $\alpha:F\ra G$ is a cofibration in
$\text{Fun}({\mathbf N}^{op}, \M)$ if it consists levelwise of cofibrations in~$\M$.
\end{example}

\section{A model approximation for relative homological algebra}
In this section we construct a Quillen pair suitable for doing
relative homological algebra with unbounded
chain complexes. The model category we propose is a tower of categories of \emph{bounded} chain complexes,
each
equipped with a relative model structure. Therefore we first define a model
structure on bounded chain complexes,
then introduce the category of towers, and finally study 
the associated Quillen pair.

\subsection{Bounded chain complexes}
Let $n$ be an integer. The full subcategory of
$\ch(\A)_{ \leq n}\subset\text{Ch}(\mathcal{A})$
consists of the chain complexes $X$ such that $X_i=0$  for $i > n$.
The inclusion functor  $\text{in}:\text{Ch}(\A)_{ \leq n}\subset \text{Ch}(\A)$
has both a right and a left adjoint.
The left adjoint is denoted by $\tau_n:\text{Ch}(\A)\ra \text{Ch}(\A)_{ \leq n}$ and is called \emph{truncation}. Explicitly
$\tau_{n}$ assigns to  a complex $X$ the truncated complex
\[
\tau_{n}(X):=(\text{coker}(d_{n+1})\xrightarrow{d_n} X_{n-1}\xrightarrow{d_{n-1}}X_{n-2}\xrightarrow{d_{n-2}}\cdots)
\]
where in degree $n$ we have $\tau_{n}(X)_{n}=\text{coker}(d_{n+1})$,
and for $i<n$ the formula is $\tau_{n}(X)_{i}=X_i$.
For a morphism $f:X\ra Y$ in $\text{Ch}(\A)$ the map $\tau_{n}(f)_n$ is  induced by $f_n$, while for $i<n$ we have  
$\tau_{n}(f)_i=f_i$.

For any $X\in \text{Ch}(\A)$,  the {\em truncation morphism} $t_n:X\ra  \text{in} \tau_{n}(X)$ is the unit of the adjunction $\tau_n\dashv\text{in}$.
Explicitly this  morphism we will abusively write as $t_n:X\ra \tau_{n}(X)$ is the following chain map:
\[
\xymatrix{
X  \ar[d]& \cdots\ar[r]^{d_{n+2}} & X_{n+1}\ar[d]\ar[r]^{d_{n+1}} & X_{n}\ar[r]^{d_{n}}
\ar[d]^{q} & X_{n-1}\ar[r]^{d_{n-1}}\ar[d]^{\text{id}} & \cdots\\
\tau_{n}(X) &\cdots\ar[r] & 0\ar[r] &\text{coker}(d_{n+1})\ar[r]^{d_{n}} &X_{n-1}\ar[r]^{d_{n-1}}& \cdots 
}
\]
where $q$ denotes the quotient morphism.
With respect to the injective classes we introduced in
Definition~\ref{def injclass},
the key property of the truncation morphism is the following.

\begin{proposition}\label{prop keycanmorp}
The truncation morphism  $t_n:X\ra \tau_{n}(X)$ is an $n$-$\I$-weak equivalence for any injective class $\I$.
\end{proposition}

\begin{proof}
For any $W \in \mathcal{I}$, the morphism $\A(t_n,W)$ is given by the following commutative diagram:
\[\xymatrix{
0\ar[d]&  & \A(\text{coker}(d_{n+1}),W)\ar[ll]\ar[d] & \A(X_{n-1},W)\ar[l]\ar[d]^{\text{id}} &&\cdots\ar[ll]_(.38){\A(d_{n-1},W)} \\
\A(X_{n+1},W) & &\A(X_{n},W)\ar[ll]_(.45){\A(d_{n+1},W)} & \A(X_{n-1},W)\ar[l]_{\A(d_{n},W)} &&\cdots\ar[ll]_(.38){\A(d_{n-1},W)} \\
}\]
Clearly $\A(t_n,W)$ induces an isomorphism on cohomology in degrees $>-n$. Since the kernel of
$\A(d_{n+1},W)$ is given by $\A(\text{coker}(d_{n+1}),W)$, $\A(t_n,W)$ induces also an isomorphism on $H^{-n}$. 
As $H^{-n-1}(\A(\tau_{n}(X),W))=0$, $\A(t_n,W)$ induces a monomorphism on~$H^{-n-1}$.
\end{proof}

We begin by recalling a theorem of Bousfield~\cite[ Section 4.4]{MR2026537}.
A proof may also be found
in the appendix, see Theorem~\ref{thm modstructureabove}---it's there
both for the reader's
convenience
and because it gives an explicit construction of fibrant replacements.

\begin{theorem}\label{thm mcnpch}
Let $\I$ be an injective class.
The following choice of weak equivalences, cofibrations and fibrations endows $\text{\rm Ch}(\A)_{ \leq n}$ with a model category structure:
\begin{itemize}
\item $f\colon X\ra Y$ is called an $\I$-weak equivalence if $f^\ast: \A(Y,W) \rightarrow \A(X,W)$ is a
quasi-isomorphism of complexes of abelian groups for any $W \in \I$.
\item $f\colon X\ra Y$ is called an $\I$-cofibration if $f_i\colon X_i \rightarrow Y_i$ is an $\I$-monomorphism  for all $i\leq n$.
\item $f\colon X\ra Y$ is called an $\I$-fibration if $f_i\colon X_i \rightarrow Y_i$ has a section and its
kernel belongs to $\I$  for all $i \leq n$. In particular $X$ is  $\I$-fibrant if $X_i \in \I$ for all $i\leq n$.
\end{itemize}
\end{theorem}

Among other things this model structure gives,
for an object $A\in\A\subset\text{\rm Ch}(\A)_{ \leq 0}$, a
fibrant replacement $A\ra I$. This turns out to be nothing other than
a relative injective resolution for $A$.

Here are some basic properties of this model structure on $\ch(\A)_{ \leq n}$.

\begin{proposition}\label{prop prpfbmcs}
\begin{enumerate}
\item All objects in $\text{\rm Ch}(\A)_{ \leq n}$ are $\I$-cofibrant.

\item Let $f\colon X\ra Y$ be an $\I$-fibration.  Then $\text{\rm Ker}(f)$ is fibrant and 
$f$ is a $k$-$\I$-weak equivalence if and only if  $\text{\rm Ker}(f)$ is $k$-$\I$-connected.

\item An $\I$-fibration $f\colon X\ra Y$ is  an $\I$-weak  equivalence if and only if
$\text{\rm Ker}(f)$ is $\I$-trivial. Moreover, if $f$ is an acyclic  $\I$-fibration, then
there is an isomorphism $\alpha\colon Y\oplus \text{\rm Ker}(f)\ra X$ for which the following diagram commutes:
\[
\xymatrix{
Y\oplus \text{\rm Ker}(f)\ar[rr]^(.55){\alpha}\ar[dr]^{\text{pr}} & & X\ar[dl]_{f}\\
& Y
}
\]

\item An  $\I$-weak equivalence between $\I$-fibrant chain complexes  in $\text{\rm Ch}(\A)_{ \leq n}$ is a homotopy equivalence. 

\item An $\I$-fibrant object   in $\text{\rm Ch}(\A)_{ \leq n}$ is $\I$-trivial if and only if it  is isomorphic to a 
complex of the form  $\bigoplus_{i\leq n} D_{i}(W_i)$, where disc complexes have been defined in Example~\ref{exs relweq}.

\item Products of $\I$-fibrant and  $\I$-trivial complexes are $\I$-trivial.
\item Assume that the following is a sequence of $\I$-fibrations and $\I$-weak equivalences in  $ \text{\rm Ch}(\A)_{ \leq n}$:
\[
(\cdots X_{2}\xrightarrow{f_2}X_{1}\xrightarrow{f_1}X_{0})
\]
Then, the projection morphism $\text{\rm lim}_{i\geq 0} X_{i}\ra X_{k}$ is an $\I$-fibration and an $\I$-weak equivalence for any $k\geq 0$.
\end{enumerate}
\end{proposition}

\begin{proof}
\noindent
(1) follows from the fact that, for any $W\in \A$,  the morphism $0\ra W$ is an $\I$-monomorphism.

\noindent
(2):\quad For any $W$, the following is an exact sequence of chain complexes of abelian groups:
\[
0\ra \A(Y,W)\xrightarrow{\A(f,W)} \A(X,W)\ra \A(\text{Ker}(f),W)\ra 0\, .
\]

\noindent
The first part of (3) follows from (2). If $f:X\ra Y$ is an acyclic $\I$-fibration, then because all objects in $\ch(\A)_{ \leq n}$
are $\I$-cofibrant, there is a morphism $s:Y\ra X$ for which $fs=\text{id}_{Y}$. This implies the second part of~(3).

\noindent
(4):\quad
All objects in  $\ch(\A)_{ \leq n}$ are $\I$-cofibrant, so an $\I$-weak equivalence between $\I$-fibrant objects
is a homotopy equivalence in the $\I$-model structure.  But, the standard path object $P(Z)$
(see~\ref{pt path}), is a very good path object for any $\I$-fibrant chain complex $Z\in\ch(\A)_{ \leq n}$
(in the terminology used in \cite{MR1361887}, which means that the factorization 
$Z\subset P(Z)\xrightarrow{\pi} Z\oplus Z$ consists in an acyclic cofibration followed by a fibration).
Hence, a homotopy equivalence in the $\I$-model structure on $\ch(\A)_{ \leq n}$ is nothing but a homotopy equivalence.

\noindent
(5):\quad
Assume that $X$ is $\I$-fibrant and $\I$-trivial. According to Proposition~\ref{prop propunch}.(8) we need to show that, 
for all $i$, $W_{i}:=\text{coker}(d_{i+1})$ belongs to $\I$.
We do it by induction on $i$. For $i=n$, $\text{coker}(d_{n+1})=X_{n}$ belongs to $\I$ since
$X$ is $\I$-fibrant. Assume now that $W_{i+1}\in\I$. As $d_{i+1}:W_{i+1}\ra X_{i}$ is an $\I$-monomorphism, it has a retraction.
It follows that $X_{i}=W_{i+1}\oplus W_{i}$. Consequently $W_{i}$, as a retract of a member of $\I$, also belongs to~$\I$.

\noindent
(6) is a consequence of (5), and (7) follows from (3) and (6).
\end{proof}

We will use the model categories $\text{\rm Ch}(\A)_{ \leq n}$ with their $\I$-relative model structure to approximate the category
of unbounded chain complexes $\ch(\A)$ equipped with the $\I$-relative weak equivalences.

\begin{proposition}\label{prop qpbcomplex}
\begin{enumerate}
\item The following pair of functors is a right Quillen pair:
\[\tau_n:\text{\rm Ch}(\A)\leftrightarrows \text{\rm Ch}(\A)_{ \leq
n}:\text{\rm in}\]
\item A chain complex $X\in \text{\rm Ch}(\A)$ is approximated by the above right Quillen pair if and only if
$\A(X,W)$ has trivial cohomology for $i<-n$ and any $W\in\I$.
\end{enumerate}
\end{proposition}

\begin{proof}
Both statements follow directly from the definitions and Proposition~\ref{prop keycanmorp}.
\end{proof}

Our aim is to find other Quillen pairs for $\text{\rm Ch}(\A)$ that approximate more unbounded chain complexes than just those 
with ``bounded $\I$-homology''. For that we construct a suitable model category by assembling the 
categories $\text{\rm Ch}(\A)_{ \leq n}$ into a tower. This is the content of the next subsection.

\subsection{Towers of bounded chain complexes}
For $n\geq k$,  the restriction of $\tau_k:\text{Ch}(\A)\ra \text{Ch}(\A)_{\leq k}$ to the subcategory 
$\text{Ch}(\A)_{\leq n}\subset  \text{Ch}(\A)$ is denoted by the same symbol
$\tau_k:\text{Ch}(\A)_{\leq n}\ra \text{Ch}(\A)_{ \leq k}$ (and is left adjoint to the inclusion
$\text{in}:\text{Ch}(\A)_{\leq k}\subset \text{Ch}(\A)_{\leq n}$). 
Moreover the canonical morphism $X\ra  \tau_{k}(X)$ can be expressed uniquely as the composite
$X\ra \tau_{n}(X)\ra \tau_{k}(X)$, of the truncation morphism $X\ra \tau_{n}(X)$ for $X$ and $n$, and  the 
truncation morphism $\tau_{n}(X)\ra \tau_{k}(X)=\tau_{k}(\tau_{n}(X))$ for $\tau_{n}(X)$ and~$k$.

Consider now the  sequence of model categories $\{\ch(\A)_{ \leq n}\}_{n\geq 0}$, with the model structures given 
by Theorem~\ref{thm mcnpch}. The  functor  $\text{in}:\ch(\A)_{ \leq n}\subset  \ch(\A)_{ \leq n+1}$ takes (acyclic) 
fibrations to (acyclic) fibrations and hence the following is a sequence of Quillen functors:
\[
\{\tau_{n}: \ch(\A)_{ \leq n+1}\rightleftarrows\ch(\A)_{ \leq n}:\text{in}\}_{n\geq 0}
\]
We will denote this tower of model categories by $\T(\A,\I)$ and  use the  symbol
$\text{Tow}(\A,\I)$ to denote the category of towers in $\T(\A,\I)$.

Let $X_{\bullet}$ be an object in $\text{Tow}(\A,\I)$. We can think about this object as a tower of morphisms:
\[\cdots\xrightarrow{t_{3}} X_{2}\xrightarrow{t_{2}}X_{1}\xrightarrow{t_{1}}X_0\]
in $\text{Ch}(\A)$ given by the  structure morphisms of $X_{\bullet}$. Conversely, for any such tower
where  $X_n$ is a chain complex that belongs to  $\ch(\A)_{ \leq
n}$, we can define an object $X_{\bullet}$ in $\text{Tow}(\A,\I)$ given by the sequence $\{X_n\}_{n\geq 0}$ 
with  the  morphisms $\{t_{n+1}\}_{n\geq 0}$  as its structure morphisms. In this way we can think about $\text{Tow}(\A,\I)$ 
as a full subcategory of the functor category $\text{Fun}({\mathbf N}^{op}, \text{Ch}(\A))$.

To be very explicit,  $\text{Tow}(\A,\I)$ is the category of commutative diagrams in $\A$ of the following form:
\begin{equation}\label{diagram}
\begin{gathered}
\xymatrix{
\vdots\ar[d] &  \vdots\ar[d] &  \vdots\ar[d] & \vdots\ar[d] & \vdots\ar[d] & \vdots\ar[d] \\
0\ar[r]& X_{2,2}\ar[r]^{d_{2,2}}\ar[d]^{t_{2,2}} & X_{2,1}\ar[r]^{d_{2,1}}\ar[d]^{t_{2,1}}& X_{2,0}\ar[r]^{d_{2,0}}\ar[d]^{t_{2,0}}& 
X_{2,-1}\ar[r]^{d_{2,-1}}\ar[d]^{t_{2,-1}}  & X_{2,-2}\ar[r]^{d_{2,-2}}\ar[d]^{t_{2,-2}}& \cdots\\
&  0\ar[r]& X_{1,1}\ar[r]^{d_{1,1}}\ar[d]^{t_{1,1}} & X_{1,0}\ar[r]^{d_{1,0}}\ar[d]^{t_{1,0}}& 
X_{1,-1}\ar[r]^{d_{1,-1}}\ar[d]^{t_{1,-1}}& X_{1,-2}\ar[r]^{d_{1,-2}}\ar[d]^{t_{1,-2}} &\cdots \\
& & 0\ar[r]& X_{0,0}\ar[r]^{d_{0,0}} & X_{0,-1}\ar[r]^{d_{0,-1}}& X_{0,-2}\ar[r]^{d_{0,-2}}&  \cdots
}
\end{gathered}
\end{equation}
where, for any $n\geq 0$ and $i\leq n$ , $d_{n,i-1}d_{n,i}=0$, i.e., horizontal lines are chain complexes.   

We will always think about   $\text{Tow}(\A,\I)$ as a model category, with the model structure given by Proposition~\ref{prop modelontowers}. 
For example, if we think about $X_{\bullet}$ 
as a tower $(\cdots\xrightarrow{t_3} X_2\xrightarrow{t_2}X_1\xrightarrow{t_1}X_0)$,
then $X_{\bullet}$ is fibrant if and only if $X_0$ is $\I$-fibrant in $\text{Ch}(\A)_{\leq 0}$ and,
for any $n\geq 0$, $t_{n+1}:X_{n+1}\ra X_n$ is an $\I$-fibration in $\text{Ch}(\A)_{\leq n+1}$.
If we think about $X_{\bullet}$  as a commutative diagram above, then $X_{\bullet}$ is fibrant if, for any $i\leq 0$, 
the objects $X_{0,i}$ belongs to $\I$, and, for any $n>0$ and $i\leq n$, $t_{n,i}$ has a section and its kernel belongs to $\I$.
Note also that since all objects in $\text{Ch}(\A)_{\leq n}$ are cofibrant, then so are all objects in $\text{Tow}(\A,\I)$.

\subsection{Alternative description}
Let us briefly outline
another way of describing the category $\text{Tow}(\A,\I)$. Consider the constant
sequence $\{\ch(\A)_{ \leq 0}\}_{n\geq 0}$ equipped with the model structure given by 
Theorem~\ref{thm mcnpch}  and the sequence of adjoint functors $\{\tau: \ch(\A)_{ \leq 0}\rightleftarrows\ch(\A)_{ \leq 0}
\colon \Sigma^{-1}\}_{n\geq 0}$, where   $\Sigma^{-1}$ is the shift functor.
It is clear that $\Sigma^{-1}$  takes (acyclic) $\I$-fibrations in $\ch(\A)_{ \leq 0}$ into (acyclic) $\I$-fibrations in $\ch(\A)_{ \leq 0}$. 
Let us denote this tower of model categories by $\T$.

Let $X_{\bullet}$ be an object in $\text{Tow}(\T)$. The structure morphisms of $X_{\bullet}$ and the differentials of the chain complexes $X_i$  
can be assembled to form a commutative diagram in $\A$ as in (\ref{diagram}).
This defines an isomorphism between the category of such commutative diagrams  and the category of towers 
$\text{Tow}(\T)$.  It then follows that  $\text{Tow}(\T)$ is also isomorphic to $\text{Tow}(\A,\I)$.

\subsection{A right Quillen pair for $\text{Ch}(\A)$}
In this subsection we use the  model category  $\text{Tow}(\A,\I)$ described above to define a right Quillen pair for  
$\text{Ch}(\A)$ that has potential to approximate more than complexes with bounded $\I$-homology (see Proposition~\ref{prop qpbcomplex}). 
We define first a pair of adjoint functors $\text{tow}:\text{Ch}(\A)\rightleftarrows \text{Tow}(\A,\I):\text{lim}$.

Let $X$ be an object in $\text{Ch}(\A)$. Define $\text{tow}(X)$ to be the object in $\text{Tow}(\A,\I)$ given by the sequence
$\{\tau_n(X)\}_{n\geq 0}$   with the structural morphisms given by the truncation morphisms $\{t_{n+1}:\tau_{n+1}(X)\ra \tau_{n}(X)\}_{n\geq 0}$.
Explicitly, $\text{tow}(X)$ is represented by the following commutative diagram in $\A$:
\begin{equation}\label{diagram2}
\begin{gathered}
\xymatrix@C=1.5em{
\vdots\ar[d]& \vdots\ar[d] & \vdots\ar[d]& \vdots\ar[d]& \vdots\ar[d]& \vdots\ar[d]\\
\tau_{2}(X) \ar[d]& 0\ar[r] & \text{coker}(d_3)\ar[r]^{d_2}\ar[d] & X_1\ar[r]^{d_1}\ar[d]^{q} & X_0\ar[r]^{d_0}\ar[d]^{\text{id}} & X_{-1}\ar[r]^{d_{-1}}\ar[d]^{\text{id}} &\cdots\\
\tau_{1}(X)\ar[d] & & 0\ar[r] &\text{coker}(d_2)\ar[r]^{d_1}\ar[d] &X_0\ar[r]^{d_0}\ar[d]^{q} &X_{-1}\ar[r]^{d_{-1}}\ar[d]^{\text{id}} &\cdots   \\
\tau_{0}(X) & & & 0 \ar[r]&  \text{coker}(d_1)\ar[r]^{d_0}& X_{-1}\ar[r]^{d_{-1}} &\cdots
} 
\end{gathered}
\end{equation}
where all $q$'s denote quotient morphisms. For a chain map $f:X\ra Y$, the morphism $\text{tow}(f)$ is given 
by the sequence of morphisms $\{\tau_n(f)\}_{n\geq 0}$.

We define next the limit functor $\text{lim}:\text{Tow}(\A,\I)\ra \text{Ch}(\A)$  to be the restriction of the
standard limit functor defined on $\text{Fun}({\mathbf N}^{op},\text{Ch}(\A))$ to the full subcategory $\text{Tow}(\A,\I)$.
Explicitly, let $X_{\bullet}$ be an object in $\text{Tow}(\A,\I)$ described by a diagram of the form (\ref{diagram}).
Then $\text{lim}(X_{\bullet})$ is the chain complex
obtained by taking inverse limits in the vertical direction:
\[\text{lim}(X_{\bullet})_{i}:=
\text{lim}(\cdots\xrightarrow{t_{3,i}}X_{2,i}\xrightarrow{t_{2,i}}X_{1,i}\xrightarrow{t_{1,i}}X_{0,i}) 
\]
and the  differential  $d_{i}:\text{lim}(X_{\bullet})_{i}\ra \text{lim}(X_{\bullet})_{i-1}$ given by $\text{lim}_{n} (d_{n,i} )$.
On morphisms, the functor  $\text{lim}:\text{Tow}(\A,\I)\ra \text{Ch}(\A)$
is defined in the analogous way by taking the inverse limits in the vertical direction. 

\begin{proposition}\label{prop rqp}
The functors  $\text{\rm tow}:\text{\rm Ch}(\A)\rightleftarrows \text{\rm Tow}(\A,\I):\text{\rm lim}$
form a right Quillen pair for $\text{\rm Ch}(\A)$ with $\I$-weak equivalences as weak equivalences.
\end{proposition}

\begin{proof}
We need to verify that the three conditions in Definition~\ref{def quillen} are fulfilled.   

\noindent
(1) We must show that the tower functor $\text{\rm tow}:\text{\rm Ch}(\A)\rightarrow \text{\rm Tow}(\A,\I)$ is left adjoint to the limit functor 
$\text{\rm lim}:\text{\rm Tow}(\A,\I)\ra\text{\rm Ch}(\A)$.

\noindent
Let $Y$ be a chain complex in $\text{Ch}(\A)$ and $X_{\bullet}$ be an object in  $\text{\rm Tow}(\A,\I)$
given by the tower $(\cdots X_2\xrightarrow{t_2}X_1\xrightarrow{t_1}X_0)$ of morphisms in $\text{Ch}(\A)$
with $X_n\in \text{Ch}(\A)_{n \geq}$.
 Consider a morphism of chain complexes
$f:Y\ra \text{lim}(X_{\bullet})$. Since $\text{lim}(X_{\bullet})$ is the  inverse limit of  the tower $X_{\bullet}$,
the morphism $f$ corresponds to a sequence of morphisms $\{f_{n}:Y\ra X_n\}_{n\geq 0}$ 
which are compatible with the structural morphisms~$t_n$.

Since the chain complex $X_n$ belongs to $\text{Ch}(\A)_{\leq n}$, the morphism $f_{n}:Y\ra X_n$
can be expressed in a unique way as a composition $Y\ra\tau_n(Y)\ra X_n$ where $Y\ra\tau_n(Y)$
is the truncation morphism.
The sequence  $\{\tau_n(Y)\ra X_n\}_{n\geq 0}$ describes a morphism $\text{tow}(Y)\ra X_{\bullet}$ in
$\text{\rm Tow}(\A,\I)$.
It is straightforward to check that this procedure defines a natural bijection from the set of morphisms between 
$Y$ and $\text{lim}(X_{\bullet})$ in  $\text{Ch}(\A)$ onto  the set of morphisms between
$\text{tow}(Y)$ and $X_{\bullet}$ in $\text{\rm Tow}(\A,\I)$.

\medskip
Condition (2) is a consequence of Proposition~\ref{prop qpbcomplex}: If $f:X\ra Y$ is an $\I$-weak equivalence in 
$\text{\rm Ch}(\A)$, then $\text{\rm tow}(f)$ is a weak equivalence in $\text{\rm Tow}(\A,\I)$.

\medskip

To prepare the proof of the third and last condition we show the following.

\medskip

\noindent
($2.5$) Let $K_{\bullet}\in \text{\rm Tow}(\A,\I)$ be a  fibrant object such that
$K_{n}$ is $\I$-trivial in $\text{\rm Ch}(\A)_{\leq n}$ for any $n\geq 0$. 
Then  $ \text{\rm lim}(K_{\bullet})$ is $\I$-trivial in $\text{\rm Ch}(\A)$.

Since  $K_{\bullet}$ is fibrant in $\text{\rm Tow}(\A,\I)$, 
$K_0$ is $\I$-fibrant in $\text{\rm Ch}(\A)_{\leq 0}$ and, for $n> 0$, the structure morphism
$t_{n}:K_{n}\ra K_{n-1}$ is an $\I$-fibration in $\text{\rm Ch}(\A)_{\leq n}$. As all $K_{n}$'s are assumed 
to be $\I$-trivial, the $\I$-fibrations $t_{n}$ are also $\I$-weak equivalences. It then follows 
from Proposition~\ref{prop prpfbmcs}.(2) that
$K_{\bullet}$ is isomorphic to the following tower of chain complexes:
\[
\cdots\ra M_{0}\oplus M_1\oplus M_2\oplus M_3\xrightarrow{\text{pr}}
M_{0}\oplus M_1\oplus M_2\xrightarrow{\text{pr}}M_{0}\oplus M_1\xrightarrow{\text{pr}}M_{0}
\]
where $M_0:=K_0$ and, for $n>0$, $M_n:=\text{Ker} \, t_n$. 
Thus $\text{lim}( K_{\bullet})\cong \prod_{n\geq  0} M_n$.  

Because $M_n$ is $\I$-trivial and $\I$-fibrant in  $\text{Ch}(\A)_{\leq n}$ Proposition~\ref{prop prpfbmcs}.(4) 
implies that  $M_n$ is isomorphic to $\bigoplus_{i\leq n}D_{i}(W_{n,i})$ for some sequence  $\{W_{n,i}\}_{i\leq n}$ of objects in $\I$. 
Substituting this to the above product describing $\text{lim}( K_{\bullet})$ we get the following  isomorphisms:
\[
\text{lim}( K_{\bullet})\cong \prod_{n\geq  0} M_n=\prod_{n\geq  0}\bigoplus_{i\leq n}D_{i}(W_{n,i})=
\prod_{n\geq  0}\prod_{i\leq n}D_{i}(W_{n,i})\cong
\prod_{i}\prod_{i\leq n}D_{i}(W_{n,i})\cong
\]
\[
\cong \prod_{i}D_{i}(\prod_{i\leq n}W_{n,i})\cong
\bigoplus_{i}D_{i}(\prod_{i\leq n}W_{n,i})
\]
It is now clear that $\text{lim}( K_{\bullet})$ is $\I$-trivial. In fact $\text{lim}( K_{\bullet})$ is even homotopy equivalent to the zero chain complex.
\medskip

\noindent
(3)  Let $f_{\bullet}:X_{\bullet}\ra Y_{\bullet}$ be a  weak equivalence in $\text{\rm Tow}(\A,\I)$ between fibrant objects. 
Then $\text{\rm lim}(f_{\bullet})$ is an $\I$-weak equivalence in $\text{\rm Ch}(\A)$.

\noindent
By Ken Brown's Lemma (see~\cite[Factorization Lemma]{MR0341469}, or \cite[Lemma~9.9]{MR1361887} for a more explicit treatment), it is enough to show the statement  under the additional 
assumption that   $f_{\bullet}:X_{\bullet}\ra Y_{\bullet}$ is an $\I$-fibration. 
Let  us define $K_{\bullet}$ to be an object in $\text{\rm Tow}(\A,\I)$ given by the sequence
$\{\text{Ker}(f_n)\}_{n\geq 0}$ with the structure morphisms being the restrictions of the structure morphisms of $X_{\bullet}$.   Since all objects in
$\text{\rm Ch}(\A)_{\leq n}$ are $\I$-cofibrant, then so are all objects in  $\text{\rm Tow}(\A,\I)$.
It follows that there is $s_{\bullet}:Y_{\bullet}\ra X_{\bullet}$ for which $f_{\bullet}s_{\bullet}=\text{id}$.
By applying the functor $\text{lim}$, we then get the following  split  exact sequence in $\text{\rm Ch}(\A)$:
\[
0\ra\text{lim}(K_{\bullet})\ra \text{lim}(X_{\bullet})\xrightarrow{\text{lim}(f_{\bullet})}\text{lim}(Y_{\bullet})\ra 0
\]
Since  $X_{\bullet}$  is isomorphic to $K_{\bullet}\oplus Y_{\bullet}$, as a retract of a fibrant object $X_{\bullet}$, the object 
$K_{\bullet}$ is then  also fibrant. Moreover, as $f_n$ is an $\I$-equivalence in $\text{\rm Ch}(\A)_{n \geq}$,  
the complex $K_n$ is $\I$-trivial in $\text{\rm Ch}(\A)_{\leq n}$ for any $n\geq 0$. 
We can then apply statement (2.5) to 
conclude that $\text{lim}(K_{\bullet})$ is an $\I$-trivial chain complex in $\text{\rm Ch}(\A)$.
The morphism $\text{lim}(f_{\bullet}):\text{lim}(X_{\bullet})\ra \text{lim}(Y_{\bullet})$ must be then
an $\I$-weak equivalence (which is in fact a homotopy equivalence).
\end{proof}

\begin{definition}\label{def standard}
{\rm The  right Quillen pair $\text{\rm tow}:\text{\rm Ch}(\A)\rightleftarrows \text{\rm Tow}(\A,\I):\text{\rm lim}$ is called the 
{\em standard Quillen pair} for $\text{\rm Ch}(\A)$.}
\end{definition}

\subsection{Complexes approximated by the standard Quillen pair}
The key task  now is to find out which chain complexes are approximated by the standard Quillen pair, i.e.~we
need to understand which complexes $X$  have the following property:
\begin{itemize}
\item If $f:\text{tow}(X)\ra Y_{\bullet}$ is a weak equivalence in $\text{Tow}(\A,\I)$, with fibrant target $Y_{\bullet}$,  then
its adjoint $g:X\ra \text{lim}(Y_{\bullet})$ is an $\I$-weak equivalence in $\text{\rm Ch}(\A)$.
\end{itemize}
Recall that if, for a chain complex $X$,  the above  statement is true for some fibrant $Y_{\bullet}$, then it is true for any other.

Assume that $f:\text{tow}(X)\ra Y_{\bullet}$ is a weak equivalence in $\text{Tow}(\A,\I)$ (for now even without the fibrancy 
assumption on $Y_{\bullet}$) and let $g:X\ra \text{lim}(Y_{\bullet})$ be its adjoint. 
Fix an integer $k\geq 0$ and  consider the following commutative diagram in $\text{Ch}(\A)$:
\[
\xymatrix{
X\ar[r]^{g}\ar[d] & \text{lim}(Y_{\bullet})\ar[d]\\
\tau_k(X)\ar[r]^{f_{k}}& Y_k
}
\]
where $\text{lim}(Y_{\bullet})\ra Y_k$ is the projection and $X\ra\tau_k(X) $ is the truncation morphism, which, 
according to Proposition~\ref{prop keycanmorp}, is an $k$-$\I$-weak equivalence.  By assumption $f_k$ is an $\I$-weak equivalence so that
the composite of $g$ with the projection $\text{lim}(Y_{\bullet})\ra Y_k$ is a $k$-$\I$-weak equivalence. 

As a consequence, the error in approximating a complex is always of a somewhat tame nature. For any $W\in \I$, $\A(g,W):\A(\text{\rm lim}(Y_{\bullet}),W)\ra\A(X,W)$ 
must induce a split epimorphism in cohomology in degrees $i\geq -k$ and this happens for all $k$'s.

\begin{proposition}\label{prop defaultissplit}
Let  $f:\text{\rm tow}(X)\ra Y_{\bullet}$ be a weak equivalence in $\text{\rm Tow}(\A,\I)$
and $g:X\ra \text{\rm lim}(Y_{\bullet})$ be its adjoint. Then,   for any $W\in \I$,
$\A(g,W)$ induces a split epimorphism on homology. \hfill{\qed}
\end{proposition}

\section{A relative version of Roos' axiom AB4*-$n$}
\label{sec ab4*rel}
In this section we show that under Roos' axiom AB4*-$n$, see \cite{MR2197371}, every complex is approximated by the standard Quillen pair.
In fact we introduce a relative version of this axiom and extend this result to provide a construction of relative resolutions for unbounded chain
complexes via towers of truncations.


\begin{definition}\label{def AB4-nrel}
Let $\mathcal{A}$ be an abelian category, $\mathcal{I}$ an injective class and $n\geq 0$ an integer. 
We say that the category $\mathcal{A}$ satisfies axiom AB4*-${\mathcal{I}}$-$n$ if and only if, 
for any countable family of objects $(A_{j})_{j \in J}$ and any choice of relative resolutions  $A_j \rightarrow I_j$, 
with $I_j \in \text{\rm Ch}(\A)_{\leq 0}$, the product complex $\prod_{j \in J} I_j$ is $(-n-1)$-$\mathcal{I}$-connected. 
\end{definition}

Roos' axiom AB4*-$n$ is stated in terms of the derived functors of products, namely that all infinite derived product functors
$\Pi^{(i)} C_\alpha$ vanish for $i>n$. Our axiom involves countable products because we only need towers indexed by the natural
integers. Except for this our axioms are closely related.

\begin{proposition}\label{prop Roos}
An abelian category $\mathcal{A}$ satisfies axiom {\rm AB4*}-${\mathcal{I}}$-$n$ for the class $\I$ of all injective objects
if and only if all derived countable product functors $\prod^{(i)} C_\alpha$ vanish for $i>n$.
\end{proposition}

\begin{proof}
Given a countable family of objects $(A_{j})_{j \in J}$, let us choose injective resolutions  $A_j \rightarrow I_j$, 
and form the product complex $\prod_{j \in J}\mathcal{I}(A_j)_\ast$. This complex is $(-n-1)$-$\mathcal{I}$-connected
if and only if it is $(-n-1)$-connected since we deal here with the class of all injectives. The higher homology of this complex 
computes the derived functors of the countable product $\prod^{(i)} A_j$. They vanish for $i>n$ precisely when the
complex is $(-n-1)$-connected.
\end{proof}

\begin{proposition}\label{thm approxunderab4}
Assume that the abelian category $\mathcal{A}$ satisfies axiom {\rm AB4*}-$\mathcal{I}$-$n$ for an injective class $\mathcal{I}$.
Let $K_\bullet$ be a fibrant tower in $\text{Tow}(\A,\I)$ such that, for any $n$, $K_n$ is $k$-$\I$-connected. 
Then the complex $\text{lim} K_\bullet$ is $(k-n-1)$-$\I$-connected.
\end{proposition}

\begin{proof}
The kernel of  the ``one minus shift'' map $1-t: \prod K_n \rightarrow \prod K_n$, defined
by $(1-t)(x_n) = (x_n - t_{n+1}(x_{n+1}))$, is $\text{lim} K_\bullet$. Since $K_\bullet$ is fibrant, 
the vertical structure maps $t_n$ are degreewise split epimorphisms and we may choose, 
in each degree, a splitting $\sigma: K_n \rightarrow K_{n+1}$. We define then maps
$\prod K_n \rightarrow K_{m+1}$ for all $m \geq 0$ by the formula
\[
(x_0, x_1,x_2, \cdots) \mapsto - \sum_{j=0}^{m}\sigma^{m+1-j}(x_{j})
\]
which assemble to form a degreewise splitting $s: \prod K_n \rightarrow \prod K_n$ of $1-t$.
This proves first that the sequence
\[
\xymatrix{
0 \ar[r] & \text{lim} K_\bullet \ar[r] & \Pi_nK_n \ar^{1-t}[r] & \Pi_n K_\bullet \ar[r] & 0\\
}
\]
is exact, and second, that applying $\mathcal{A}(-,W)$ for any $W \in \I$ to the previous sequence gives an exact sequence of 
complexes, which is also split in each degree:
\[
\xymatrix{
0 & \mathcal{A}(\text{lim} K_\bullet,W) \ar[l] & \mathcal{A}(\prod_n K_n,W) \ar[l] & \mathcal{A}(\prod_n K_n,W) \ar^{1-t}[l] & 0 \ar[l] \\
}
\]
Therefore, any bound on the connectivity of $\mathcal{A}(\prod_n K_n,W)$ is a bound on the connectivity of $\mathcal{A}(\text{lim} K_\bullet,W)$. 
We will conclude the proof by showing that $k-n-1$ is such a bound.
Observe now that, because each complex $K_m$ is $k$-$\I$-connected, the following sequence is exact:
\[
\xymatrix{
\mathcal{A}((K_{m})_{k+1},W)  & \mathcal{A}((K_{m})_{k},W)\ar[l]  & \mathcal{A}((K_{m})_{k-1},W) \ar[l] & \ar[l] \dots
}
\]
Left exactness of the functor $\mathcal{A}(-,W)$ shows that the kernel of the leftmost arrow above is 
$\mathcal{A}((K_{m})_{k}/(K_{m})_{k+1},W)$. In particular, as the complex $K_m$ is $k$-connected and fibrant, 
the truncated complex $\tau_k(K_m)$ yields a (shifted) relative $\mathcal{I}$-resolution of the object $(K_{m})_{k}/(K_{m})_{k+1}$.

Hence, in computing in degree $q <k-n-1$ the cohomology of the complex $\mathcal{A}(\prod K_n,W)$ we are computing, 
under the axiom AB4*-$\I$-$n$, the cohomology in degree $<-n$ of an acyclic complex. 
\end{proof}

And finally we get our expected approximation:

\begin{theorem}\label{thm ab4igivesapprox}
Let $\I$ be an injective class and assume that the abelian category $\mathcal{A}$ satisfies axiom {\rm AB4*}-$\mathcal{I}$-$n$.
Then the standard Quillen pair
\[
\text{\rm tow}:\text{\rm Ch}(\A)\rightleftarrows \text{\rm Tow}(\A,\I):\text{\rm lim}
\]
is a model approximation.
\end{theorem}
\begin{proof}
In view of Proposition~\ref{prop rqp} it remains to show that to any fibrant replacement $f:\text{tow}(X) \rightarrow Y_\bullet$ 
corresponds an adjoint $X \rightarrow \textrm{lim} Y_\bullet$ that is an $\I$-weak equivalence.
Let $\{t_{n+1}:Y_{n+1}\ra Y_{n}\}_{n\geq 0}$ be  the structure morphisms of $Y_{\bullet}$ and,
for any $n\geq  k$, let  $c_{n-k}:{Y_{n}}\ra Y_k$ denote the composite
\[
\xymatrix{
Y_{n}\ar[r]^{t_{n-k}} 
&Y_{n-1}\ar[r]^{t_{n-1}} &\cdots\ar[r]^{t_{k+1}} &Y_{k}
}
\]
These morphisms fit into the following commutative square in $\text{Ch}(A)_{\leq n}$ where the top horizontal morphism is 
a truncation morphism:
\[
\xymatrix{
\tau_{n}(X)\ar[r]\ar[d]_{f_n} & \tau_k(X)\ar[d]^{f_k}\\
Y_n\ar[r]^{c_{n-k}} & Y_k
}
\]
By assumption, $f_n$ and $f_k$ are $\I$-weak equivalences, and the top horizontal arrow is a $k$-$\I$-weak equivalence
according to Proposition~\ref{prop keycanmorp}. It follows that $c_{n-k}$ is a $k$-$\I$-weak equivalence.
Fibrancy of $Y_{\bullet}$ implies that $c_{n-k}$ is also an $\I$-fibration in 
$\text{Ch}(A)_{\leq n}$. In particular, for $n\geq k$, $K_{n}:=\text{ker}(c_{n-k}:Y_{n}\ra Y_{k})$ is  $k$-$\I$-connected 
(see Proposition~\ref{prop prpfbmcs}.(2)). Set $K_{n}:=0$ for $n<k$, and 
define  $t_{n+1}:K_{n+1}\ra K_{n}$ to be the restriction of the structure morphism $t_{n+1}:Y_{n+1}\ra Y_{n}$, if $n\geq k$, and  the zero morphism, if $n<k$.
In this way we have defined a fibrant object  $K_{\bullet}$ in $\text{Tow}(\A,\I)$. 
We have moreover a degreewise  split exact sequence  in $\ch(\A)$:
\begin{equation}\label{diagram3}
\xymatrix{
0\ar[r]& \text{lim}(K_{\bullet}) \ar[r] & \text{lim}(Y_{\bullet})\ar[r] & Y_{k}\ar[r]&  0 
}
\end{equation}

By Proposition~\ref{thm approxunderab4}, $ \text{lim}(K_{\bullet})$ is $(k-n-1)$-$\I$-connected, hence 
$\text{lim}(Y_{\bullet}) \ra Y_{k}$ is a $(k-n-1)$-$\I$-equivalence. 
But since the composite $X = \text{lim} (\tau_n(X)) \rightarrow \lim(Y_\bullet) \rightarrow Y_k$
is a $k$-$\mathcal{I}$-weak equivalence, it follows by the $2$ out of $3$ property that $X \rightarrow \text{lim}(Y_{\bullet})$ 
is a $(k-n-1)$-$\mathcal{I}$-weak equivalence.

This is so for any value of $k$, which concludes the proof.
%
\end{proof}

This explains also why Spaltenstein's construction of resolutions via truncations works in the absolute setting.

\begin{corollary}
\label{cor:Spaltenstein}
Let $R$ be any ring and $\I$ be the class of all injective $R$-modules. The category of $R$-modules satisfies
axiom {\rm AB4*}-$\mathcal{I}$-$0$. In particular the standard Quillen pair
\[
\text{\rm tow}:\text{\rm Ch}(\A)\rightleftarrows \text{\rm Tow}(\A,\I):\text{\rm lim}
\]
is a model approximation.
\end{corollary}

\begin{proof}
Relative connectivity for the class of all injective modules is connectivity and the category of $R$-modules satisfies
axiom AB4*, which is AB4*-$0$ as stated in \cite[Remark~1.2]{MR2197371}.
\end{proof}

\section{Example: Noetherian rings with finite Krull dimension}
\label{sec:finiteKdim}
In this section $R$ is a Noetherian ring and we focus on injective classes of injectives, which where classified in \cite{CPSinj}: 
they are in one-to-one correspondence with the generization closed subsets of Spec$R$. 
We show that, under the added assumption that $R$ is of finite
Krull dimension, the standard 
Quillen pair is a model approximation
for all injective classes $\I$ of injectives.

We need some preparation before proving this theorem, and
we refer to Appendix~\ref{sec:algebra}
for elementary facts about local cohomology. The key
ingredient is the vanishing of the homology of an $\I$-relative resolution above the Krull dimension of the ring.

\begin{lemma}
\label{lemma complex at p}
Let $R$ be a Noetherian ring,
${\mathfrak p}\subset R$ a prime ideal of height~$d$, and
$I \in \text{Ch}(R)_{\leq 0}$ an
injective resolution of a module $M$. The complex $I({\mathfrak p})$, obtained
from $I$ by keeping only the direct summands
isomorphic to $E(R/{\mathfrak p})$, has no homology in
degrees $< -d$.
\end{lemma}

\begin{proof}
  First form $I \otimes R_{\mathfrak p}$, that is
  localize $I$ at ${\mathfrak p}$ to kill all the summands
  of $I$ isomorphic to $E(R/{\mathfrak q})$ with
  ${\mathfrak q} \not\subset {\mathfrak p}$, see
Lemma~\ref{lemma local properties of injectives}. The subcomplex
$\Gamma_{\mathfrak p}(I \otimes R_{\mathfrak p})$ is precisely
what we obtain from $I \otimes R_{\mathfrak p}$ by excising
summands isomorphic to $E(R/{\mathfrak q})$ for
  ${\mathfrak q} \subsetneq {\mathfrak p}$,
see
Lemma~\ref{lemma p-torsion injective}. Thus 
$\Gamma_{\mathfrak p}(I \otimes R_{\mathfrak p})=I({\mathfrak p})$,
 in the notation of the current Lemma.
The ring $R_{\mathfrak p}$ is flat over $R$, hence $I \otimes R_{\mathfrak p}$
is an injective resolution over $R_{\mathfrak p}$ of the module
$M \otimes R_{\mathfrak p}$, and the cohomology of
$I({\mathfrak p})=\Gamma_{\mathfrak p}(I \otimes R_{\mathfrak p})$
is the local cohomology of
$M \otimes R_{\mathfrak p}$ at the maximal ideal
${\mathfrak p}R_{\mathfrak p}
\subset R_{\mathfrak p}$.
 The vanishing 
follows from Proposition~\ref{proposition Cech} and Remark~\ref{remark
uptonilpotence} because the Krull dimension of $R_{\mathfrak p}$ is~$d$.
\end{proof}

\begin{proposition}
\label{proposition homology bound}
Let $R$ be a Noetherian ring of finite Krull dimension~$d$,
and $\I$ an injective class of injective modules. For any module $M$
and an $\I$-relative resolution $I \in \text{Ch}(R)_{\leq 0}$, we have
$H_k(I) = 0$ if $k< -d-1$.
\end{proposition}

\begin{proof}
The injective class $\I$ corresponds to a generization closed subset $S$ of $Spec(R)$
by \cite[Corollary~3.1]{CPSinj}. Let $a$ be the length of the maximal
chain of prime ideals in the complement of $S$. If $a=0$ then $\I$
consists of all injective modules, so that $H_k(I) = 0$ for all
$k <0$.

Assume now that $a\geq1$ and we prove the Proposition by induction
on $a$. Consider the set $S''$ of minimal ideals ${\mathfrak p}_i$
in the complement of $S$;
we know the result is true for the injective class $\I'$
corresponding to the set $S'= S \cup S''$. We denote by $I'$ the
$\I'$-relative resolution of $M$. Replacing $I$ by a homotopy equivalent
complex if necessary, we obtain a degree-wise split short exact sequence of
chain complexes $0\ra I''\ra I'\ra I\ra 0$, where $I''$ is a direct sum
of $E(R/{\mathfrak p}_i)$ with ${\mathfrak p}_i\in S''$.
But there are no inclusions among the primes ${\mathfrak p}_i\in S'-S$,
hence $I''$ is the direct sum of the complexes $I({\mathfrak p}_i)$ that we
introduced in Lemma~\ref{lemma complex at p}. The Proposition now follows
from Lemma~\ref{lemma complex at p} and the long exact sequence in homology
induced by the short exact sequence of complexes $0\ra I''\ra I'\ra I\ra 0$.
\end{proof}

Next comes the last proposition we will use in the proof of our main
theorem, it measures the difference between the resolutions of a
bounded complex and of a truncation. Recall that $I(X)$ denotes the
fibrant replacement of the bounded complex $X$ in the $\I$-relative
model structure described in Theorem~\ref{thm mcnpch}, i.e. an
$\I$-relative injective resolution of $X$.

\begin{proposition}
\label{proposition difference}
Let $R$ be a Noetherian ring of finite Krull dimension~$d$,
and $\I$ an injective class of injective modules. Let $X \in
\text{Ch}(R)_{\leq 0}$ be a bounded complex and $\tau_1X$ its first
truncation. Then the canonical morphism $X \ra \tau_1X$ induces
isomorphisms in homology $H_k(I(X)) \ra H_k(I(\tau_1 X))$ for any
$k < -d-1$.
\end{proposition}

\begin{proof}
Let us replace $X \ra \tau_1 X$ by an $\I$-fibration $I(X) \ra
I(\tau_1 X)$ between $\I$-fibrant objects. The kernel $K$ is a chain
complex made of injective modules in $\I$, and forms therefore an
$\I$-fibrant replacement for $H_0(X)$, the kernel of the canonical
morphism.

From the previous proposition we know that $H_k(K) = 0$ if $k< -d-1$.
The long exact sequence in homology
finishes the proof.
\end{proof}

\begin{theorem}
\label{theorem finite Kdim}
Let $R$ be a Noetherian ring with finite Krull dimension~$d$, and
$\I$ an injective class of injective modules. Then the category of towers
forms a model approximation for $\text{Ch}(R)$ equipped with
$\I$-equivalences.
\end{theorem}

\begin{proof}
To show that the Quillen pair is in fact a model approximation, we
must check that Condition~(4) of Definition~\ref{def quillen} holds,
or equivalently that the canonical morphism $\text{lim} \, I(\text{tow}
X) \rightarrow X$ is an $\I$-equivalence for any unbounded chain
complex~$X$. We have learned from Proposition~\ref{proposition
difference} that the homology of $I(\tau_n X)$ and $I(\tau_{n-1}
X)$ only differ in degrees lying between $n$ and $n-d-1$. This means
that the homology of the $\I$-fibrant replacement of the tower
$\text{tow}(X)$ stabilizes. Therefore $H_k(\text{lim}, I(\text{tow}
X)) \cong H_k(I(\tau_{k+d+1} X))$.
\end{proof}

\begin{remark}
\label{rem:finite Kdim}
The above argument actually shows that the category of $R$-modules satisfies axiom AB4*-$\I$-$(d+1)$ when
$R$ has finite Krull dimension $d$, and
$\I$ is an injective class of injective $R$-modules. A product of relative injective
resolutions of certain $R$-modules is
a special case of an inverse limit of a fibrant tower as above.
\end{remark}

\section{Example: Nagata's ``bad Noetherian ring"}
\label{sec:ring}
The objective of this section is to show that, even under the
Noetherian assumption, towers do not always approximate unbounded
chain complexes. We have seen in the previous section that no
problems arise when the Krull dimension is finite. However,
delicate and interesting issues arise when the
Krull dimension is infinite. 
We first recall an
example of Noetherian ring with infinite Krull dimension,
constructed by Nagata in the appendix of \cite{MR0155856}.

\begin{example}
\label{ex:badring}
{\rm Let $k$ be a field and consider the polynomial ring on
  countably many variables $A = k[x_1, x_2, \dots ]$. Consider the
  sequence of prime
ideals ${\mathfrak p}_2 = (x_1, x_2)$, ${\mathfrak p}_3 = (x_3, x_4, x_5)$, ${\mathfrak p}_4 = (x_6, x_7, x_8)$, etc.
where the depth of ${\mathfrak p}_i$ is precisely $i$. Take $S$ to be the
multiplicative set consisting of elements of $A$ which are not in
any of the ${\mathfrak p}_i$'s. The localized ring $R = S^{-1}A$ is Noetherian,
but of infinite Krull dimension. In fact its maximal ideals are $\mathfrak{m}_i
= S^{-1}{\mathfrak p}_i$, a sequence of ideals of strictly increasing height.}
\end{example}

\subsection{A problematic class of injectives}
\label{subsec:problem}
In this subsection we choose the specialization closed subset $C$ of
$\text{Spec}(R)$ to consist
of all the maximal ideals $\mathfrak{m}_i$.
We will do relative homological algebra with respect to the
injective class $\I$ of injective $R$-modules, generated by the
injective envelopes $E(R/{\mathfrak p})$ for all prime ideals ${\mathfrak p} \notin C$. We noticed earlier that the class of $\I$-acyclic
chain complexes is a localizing subcategory of $D(R)$. As it
contains $R/\mathfrak{m}_i$ but not any other $R/{\mathfrak p}$, we know from Neeman's
classification~\cite{MR1174255} that this localizing subcategory is
generated by $\oplus R/\mathfrak{m}_i$.


\begin{lemma}
\label{lemma resolution for R}
Let $I(R)$ be an $\I$-injective resolution of $R$. Then $H_0(I(R)) \cong R$
and $H_{1-i}(I(R)) \cong E(R/\mathfrak{m}_i)$ for any $i>1$.
\end{lemma}

\begin{proof}
Let us consider a minimal injective resolution $R \hookrightarrow I_0=E(R) \rightarrow I_{-1} \rightarrow \dots$. 
By the description Matlis~ \cite{MR0099360}
gave of injective modules, each $I_n$ is a direct
sum of modules of the form $E(R/{\mathfrak p})$ where ${\mathfrak p}$ runs over prime ideals of~$R$.

By Lemma~\ref{lemma hom between injectives} we see that there is a
subcomplex $K$ of $I$ made of all the copies of $E(R/\mathfrak{m}_i)$,
and we take $I(R)
= I/ K$. This is a fibrant replacement for $R$ in the relative model
structure described in Theorem~\ref{thm mcnpch}.
%
Since the homology of $I$ is concentrated in degree~$0$, we see
from the long exact sequence in homology for the short exact
sequence of complexes $K \rightarrow I \rightarrow I(R)$ that the
lower homology modules of $I(R)$ are isomorphic to those of $K$ up to a
shift: $H_{1-i} (I(R)) \cong H_{-i}(K)$ for $i>1$. 
But $K$ splits as a direct sum $\oplus_i \Gamma_{\mathfrak{m}_i}(I)$ by
Lemma~\ref{lemma hom between injectives} and Definition~\ref{def local}. Therefore
\[
H_{-k}(K) \cong \oplus H_{\mathfrak{m}_i}^{k}(R) \cong \oplus_i H_{\mathfrak{m}_i}^{k}(R_{\mathfrak{m}_i})
\]
where the second isomorphism comes from Lemma~\ref{lemma p-torsion injective}.
The local ring $R_{\mathfrak{m}_i}$ is regular, hence Gorenstein, of
dimension~$i$. Therefore the computation done in
\cite[Theorem~11.26]{MR2355715} yields that 
$H_{1-i}(I(R)) \cong E(R/\mathfrak{m}_i)$. It also shows here that all local cohomology 
modules are zero in degree zero. Hence $H_0(I(R)) \cong H_0(I) \cong R$. 
\end{proof}

Now we consider the unbounded chain complex $X$ with $X_n = R$ for
all $n$ and zero differential. The zeroth truncation of $X$ is the
non-positively graded complex with zero differential and where every
module is $R$, in other words this complex is $\oplus_{i\leq 0}
\Sigma^iR$. We know how to construct explicitly an $\I$-relative
resolution for this bounded complex by the previous lemma: it is a
direct sum $\oplus_{i\leq 0} \Sigma^i I(R)$.

\begin{lemma}
\label{lemma resolution for X}
Let $X$ be the unbounded complex $\oplus_i \Sigma^i R$, let $\tau_0 X$
be its zeroth truncation, and let $I(\tau_0 X)$ denote the $\I$-relative
resolution of the latter. We have then $H_{1-i}(I(\tau_0 X)) \cong R \oplus
\oplus_{2\leq j\leq i} E(R/\mathfrak{m}_j)$ for any $i \geq 1$.
\end{lemma}

\begin{proof}
This is a direct consequence of the previous lemma.
\end{proof}

The unbounded complex $X$ is the key player in our main counterexample.

\begin{theorem}
\label{theorem counterexample}
For  Nagata's ring $R$ and the injective class $\I$ above,
the category of towers $\hbox{\rm Tow}(R, \I)$ does not form a
model approximation for $\ch(R)$. 
More precisely there exists a complex $X$ which is not $\I$-weakly equivalent to
the limit of the fibrant replacement of its truncation tower.
\end{theorem}

\begin{proof}
The complex $X$ is the one we have constructed above, namely
$\oplus_{i \in \mathbf Z} \Sigma^i R$. 
Let us consider its tower approximation, which is, by
definition, the limit $Y$ of the tower given by the $\I$-relative
resolution of the successive truncations of $X$. From the previous lemma the
$n$th level of this tower is $\oplus_{i\leq n} \Sigma^i I(R)$ and
the structure maps are the projections. Therefore the limit is the
product $\prod_i \Sigma^i I(R)$. In particular we identify for any $i$
\[
H_{1-i}(Y) \cong R \times \prod_{j \geq 2} E(R/\mathfrak{m}_j).
\] 
The homotopy fiber of the natural map $X \rightarrow Y$ 
is thus an unbounded complex whose homology is $\prod_{j \geq 2} E(R/\mathfrak{m}_j)$ in each degree. 
This complex cannot be $\I$-acyclic since the annihilator of the image of $1$ via the
(diagonal) composite map
\[
R \rightarrow \prod_j R \rightarrow \prod_j R/\mathfrak{m}_j \rightarrow \prod_j E(R/\mathfrak{m}_j)
\]
is zero and this contradicts the description of $\I$-acyclic complexes given in Example~\ref{ex injcall-stideal}.
\end{proof}

\subsection{Well behaved classes of injectives}
\label{subsec:noproblem}
Nagata's ring, or other Noetherian rings of infinite Krull dimension, also have well behaved classes of injective modules. 
Let us fix for example a maximal ideal $\mathfrak{m}$ of height $n$.
Since the set of primes strictly contained in $\mathfrak{m}$ is saturated by \cite{CPSinj} we may consider the 
injective class $\I_{\mathfrak{m}}$ generated by $\{E(R/{\mathfrak p}) \ | \ {\mathfrak p} \varsubsetneq \mathfrak{m}\}$.

\begin{theorem}\label{thm goodinbad}
The category $R$-Mod satisfies axiom {\rm AB4*}-$\I_\mathfrak{m}$-$(n+1)$, where $n$ is 
$\textrm{height } \mathfrak{m}$. In particular the category of towers $\hbox{\rm Tow}(R, \I_{\mathfrak m})$ is
a model approximation for $\ch(R)$.
\end{theorem}

\begin{proof}
  Let $X$ be an object in $\ch(\A)_{ \leq 0}$, let $I$ be an
  injective resolution for $X$, and let $I(X)$ be the $\I$--fibrant replacement
 of $X$ obtained by excising all the summands of $I$ isomorphic
  to $E(R/\mathfrak{q})$ 
  for $\mathfrak{q}$ not strictly contained in $\mathfrak{m}$. We have a short
exact sequence of chain complexes $0\ra K\ra I\ra I(X)\ra 0$, with
$K$ a complex of injectives all of which are direct sums of
$E(R/\mathfrak{q})$ 
for $\mathfrak{q}$ not strictly contained in $\mathfrak{m}$.
Since $I(X)$ is a complex of $\mathfrak{m}$--local modules, 
tensoring with $R_\mathfrak{m}$ gives the exact sequence 
\[
\xymatrix{
0 \ar[r] & K\otimes R_\mathfrak{m} \ar[r] & I\otimes R_\mathfrak{m} \ar[r] & I(X) \ar[r] & 0.
}
\]
The first complex is a complex of injectives, each of
which is a direct sum of injectives of the form $E(R/\mathfrak m)$.
Thus over the ring $R_\mathfrak{m}$, the complex 
$I(X)$ can be viewed as the fibrant replacement of $I\otimes R_\mathfrak{m}$
with respect to the injective class of injectives
$\I'=\I\cap\text{Spec}(R_\mathfrak{m})$.
But this reduces us to the case of the noetherian local ring
$R_\mathfrak{m}$ which is of finite Krull dimension.
Theorem~\ref{theorem finite Kdim} finishes the proof.
\end{proof}

\section{Further examples}
In this section we gather some other examples of relative homological algebra settings that may be found across the literature 
and show how they tie back to our framework.

\subsection{Some Grothendieck categories studied by Roos}
The original work of Roos is precisely about finding a way to deal with the failure of axiom AB4*. He provides a 
nice and elementary example of a Grothendieck category that satisfies axiom AB4*-$n$ but not AB4*-$(n-1)$. 
This example is very close in spirit to our study of injective classes of injectives for the category of modules over a ring of  
finite Krull dimension in Section~\ref{sec:finiteKdim}. 

\begin{proposition}\cite[Theorem~1.15]{MR2197371}
The Grothendieck category $Qcoh$ of quasicoherent sheaves on the complement of the maximal ideal $\mathfrak{m}$  of 
the spectrum of a local Noetherian ring $R$  satisfies  condition {\rm AB4*}-$n$ where $n = \max(\dim (R) - 1, 0)$, 
and no lower value of $n$ is possible.
\end{proposition} 

It turns out that injective classes of injectives on Grothendieck categories correspond to the so called hereditary torsion theories. 
Building on this observation, Virili recently investigated whether Roos' axiom AB4*-$n$ holds in localizations of Grothendieck 
categories with respect  to these hereditary torsion theories. The answer depends then on the Gabriel dimension of the localized category, 
a generalization of the Krull dimension to Grothendieck categories due to Gabriel. We refer to Virili's paper~\cite{Virili} for the precise statements.

 \subsection{Pure injective classes}
Purity is a vast subject, of which we will only present the (very) thin part that is directly related to our framework. 
As a general reference one could consult M.~Prest~\cite{MR2530988}, but let us recall the basic definitions.

Let $R$ be a ring, a morphism of $R$-modules   $f:M \rightarrow N$ is said to be \emph{pure} if and only if for any $R$-module $L$, 
$f\otimes id_L: M \otimes L\rightarrow N \otimes R$ is injective. 
Then a \emph{pure-injective module} (a.k.a. algebraically compact) is an $R$-module $W$ such that for any pure homomorphism $f$, 
the induced map $\Hom(f,W)$ is surjective. A product of pure-injectives is again pure-injective and module categories have enough pure-injectives \cite{MR2530988}. Thus, pure-injective modules form an injective class as defined in Definition~\ref{def injclass}.

The following theorem shows that rings of small
cardinality satisfies a very strong version of the relative AB4* axiom with respect to the injective class 
of pure-injectives: all objects are of finite pure-injective dimension.

\begin{theorem}[Kielpinski-Simson\cite{MR0407089}, Gruson-Jensen\cite{MR0320112}]
Let $R$ be a ring of cardinality $\aleph_t$, with $t \in \mathbb{N}$. Then the pure global dimension  of $R$ is $\leq t+1$.
\end{theorem}

Applying this to our framework we obtain immediately the analogous result to Theorem~\ref{theorem finite Kdim}.

\begin{corollary}
Let $R$ be a ring of cardinality $\aleph_t$ with $t \in \mathbb{N}$, and let $\mathcal{PI}$ denote the class of pure-injective modules. 
Then the standard Quillen pair
\[
\text{\rm tow}:\text{\rm Ch}(R)\rightleftarrows \text{\rm Tow}(R,\mathcal{PI}):\text{\rm lim}
\]
is a model approximation. \hfill{$\square$}
\end{corollary}

\subsection{Gorenstein homological algebra}

This is again a vast and very active research subject, for which we refer for instance to Enochs-Jenda~\cite{MR1753146} and Holm~\cite{MR2038564}.

Given a ring $R$ an $R$-module $E$ is said to be \emph{Gorenstein injective} if there exists an exact complex of injective modules
\[
\xymatrix{
I_ \bullet: \cdots  \ar[r] & I_2 \ar[r] & I_1 \ar[r] & I_0 \ar[r] & I_{-1} \ar[r] & \cdots
}
\]
such that for any injective module $J$ the complex $\Hom(J,I_\bullet)$ is acyclic and $E= \ker(I_0 \rightarrow I_{-1})$. Denote the class of
Gorenstein injective modules by $\mathcal{GI}$. We learn in \cite[Theorem~2.6]{MR2038564} that $\mathcal{GI}$ 
contains all injective modules, and that it is closed under arbitrary products and under direct summands.

The existence of enough Gorenstein injectives (a.k.a. Gorenstein injective pre-envelopes) for general modules is more problematic. Nevertheless 
Holm shows in \cite[Theorem 2.15]{MR2038564} that any $R$-module of finite Gorenstein injective dimension admits a Gorenstein injective 
pre-envelope and thus a Gorenstein injective resolution in the sense of the present work (or ``coproper right Gorenstein injective resolution" in 
Holm's terminology). Enochs and Lopez-Ramos prove also in~\cite{MR1926201} that there are enough Gorenstein injectives in any Noetherian ring.

\begin{proposition}[\cite{MR1926201}]
The class $\mathcal{GI}$ of Gorenstein injective modules is an injective class for any Noetherian ring.
\end{proposition}

If we wish to ensure that there are enough Gorenstein injectives, it is therefore enough to assume that all modules have finite 
Gorenstein injective dimension. It would be interesting to have conditions ensuring that for a given ring the relative version of axiom 
AB4*-$n$ is satisfied, but for the moment we confine ourselves to the stronger condition that there is a bound on the Gorenstein 
injective dimension of all modules. By Enochs-Jenda~\cite{MR1753146} this characterizes Gorenstein rings. As
above we readily deduce the following proposition. The (finite) dimension of the ring is the natural number $n$ such that
the category of $R$-modules satisfies AB4*-$\mathcal{GI}$-$n$.

\begin{proposition}
Let $R$ be a Gorenstein ring. Then the standard Quillen pair
\[
\text{\rm tow}:\text{\rm Ch}(R)\rightleftarrows \text{\rm Tow}(R,\mathcal{GI}):\text{\rm lim}
\]
is a model approximation. \hfill{$\square$}
\end{proposition}

\appendix
\section{Relative homological algebra for left bounded complexes}
\label{app left}

In this section we work in a fixed abelian category $\A$, and we fix an injective class $\I$, as in Definition~\ref{def injclass}. 
In particular we assume that there are enough relative injectives. We want to show that one can equip $\ch_{\leq 0}(\A)$
with an $\I$-relative Quillen model structure. We basically  follow Quillen's arguments in \cite{MR0223432}. We will use
the terminology ($\I$-cofibrations, $\I$-weak equivalences, etc.) as introduced in Theorem~\ref{thm mcnpch}.
Before going into the homotopical subtleties, let us recall  a couple of standard
of constructions.

\begin{point}{ \bf The cone construction.}
\label{point cone}
Let $X$ be a chain complex in $\ch_{\leq 0}(\A)$. Define a complex
$CX$ as follows: $CX_0 = X_{-1}$ and for any $n< 0$ $CX_{n} =
X_{n}\oplus X_{n-1}$. The differential $CX_0 \rightarrow CX_{-1}$
is $(Id, d)$ and the lower ones $X_{n} \oplus X_{n-1}
\rightarrow X_{n-1} \oplus X_{n-2}$ are given in matrix form by
\[
\left[ \begin{matrix} 
d & (-1)^{n} Id \\ 
0 & d  
\end{matrix}\right]
\]
There is a natural chain map $X \rightarrow CX$ given by the inclusion on the first factor, except in degree zero
where we use the differential.
\end{point}

\begin{lemma}\label{lemma cone}
The cone $CX$ of any complex $X \in \ch_{\leq 0}(\A)$ is acyclic.
The chain map $X \rightarrow CX$ is a split injection in strictly negative degrees, so in particular
an $\I$-cofibration. \hfill{$\square$}
\end{lemma}

\begin{point}{\bf The mapping cylinder.}
\label{point cylinder}
Let $f: N \rightarrow M$ be a morphism of left bounded chain complexes.
Denote by $\partial$ and $d$ respectively the differentials of the
complexes $N$ and $M$. We define a new complex $\hbox{\textrm Cyl}(f)$ as follows : $\hbox{\textrm Cyl}(f)_0 = N_0 \oplus M_0$ 
and for $i <0$ $\hbox{\textrm Cyl}(f)_{i} = N_{i} \oplus M_{i+1} \oplus M_{i}$.

The differentials are given as follows
\[
\xymatrix{
\hbox{\textrm Cyl}(f)_{i} = \ar[d] & N_{i} \ar[d]^{\partial} \ar[drr]^{(-1)^{i-1}f}& & M_{i+1} \ar[d]^d && M_{i} \ar[dll]_{(-1)^{i}Id} \ar[d]^d \\
\hbox{\textrm Cyl}(f)_{i-1} = & N_{i-1} && M_{i} && M_{i-1} }
\]

We have a level-wise split injection $N \rightarrow \hbox{\textrm Cyl}(f)$ given by 
$(Id, f)$ whose cofiber is acyclic. The splitting is given by the projection on the first factor $\hbox{\textrm Cyl}(f) \rightarrow N$ (a
chain map). We have also a level-wise  split epimorphism $\hbox{\textrm Cyl}(f) \rightarrow M$, given by the projection on the last factor.
This shows the following lemma.
\end{point}

\begin{lemma}\label{lem cylinder}
The factorization $N \rightarrow \hbox{\rm Cyl}(f) \rightarrow M$ consists in a trivial $\I$-cofibration followed by a degreewise split epimorphism.
 \hfill{$\square$}
\end{lemma}

\begin{point}{\bf Path object.}
\label{pt path}
The  path object of  $X\in\text{Ch}_{\leq 0}(\mathcal{A})$ is a chain complex $\hbox{\rm P}(X)\in\text{Ch}_{\leq 0}(\mathcal{A})$  where 
$\hbox{\rm P}(X)_i:=X_i\oplus X_i\oplus X_{i+1}$ and the differential $d_i:\hbox{\rm P}(X)_i\ra \hbox{\rm P}(X)_{i-1}$ is given by the matrix:
\[d_i=\left[
\begin{array}{ccc}
d_i & 0 & 0\\
0& d_i & 0\\
(-1)^{i+1} & (-1)^{i} & d_{i+1}
\end{array}
\right]
\]
The projections $\hbox{\rm P}(X)_i=X_i\oplus X_i\oplus X_{i+1}\xrightarrow{\text{pr}} X_i\oplus X_i$  define a morphism
$\pi:\hbox{\rm P}(X)\ra X\oplus X$ and the diagonals  $(\text{id},\text{id}, 0):X_i\ra   X_i\oplus X_i\oplus X_{i+1}=\hbox{\rm P}(X)_i$ 
define a morphism $h:X\ra \hbox{\rm P}(X)$. The  factorization 
$X\xrightarrow{h} \hbox{\rm P}(X)\xrightarrow{\pi}X\oplus X$, of the diagonal $(\text{id},\text{id}):X\ra X\oplus X$, 
will be also called the standard path object of $X$.
The path object is a functorial construction and it  commutes with arbitrary products and coproducts.

Two morphisms $f,g:X\ra Y$ in $\text{Ch}_{\leq 0}(\mathcal{A})$ are homotopic
if there is a sequence of morphisms $s_i:X_i\ra Y_{i+1}$ such that
$f_i-g_i=d_{i+1}s_i+ s_{i-1} d_i$ for any $i$. This is equivalent to the existence of a morphism
$h:X\ra \hbox{\rm P}(Y)$   for which the composition $X\xrightarrow{h} \hbox{\rm P}(Y)\xrightarrow{\pi} Y\oplus Y$
is given by $(f,g):X\ra Y\oplus Y$.
If $f$ and $g$ are homotopic, then $H_i(f)=H_i(g)$.
\end{point}

%
%


\begin{point}{\bf Factorization axioms.}
\label{point factorization}
We need a few preliminary results to construct the $\I$-relative factorizations.
The property of $\A$ having enough $\I$-injectives can be extended
to the following property of $\ch(\A)$. We do not claim any
functoriality in this statement, as there are many choices
involved in the construction.
\end{point}

\begin{lemma}\label{lemma reschains}
If $\A$ has enough $\I$-injectives, then for any chain complex $X \in \ch(\A)$ there exists a map of chain complexes 
$X \rightarrow I$ such that $I_i \in \I$ and $X_i \rightarrow I_i$ is an $\I$-monomorphism for any $i$. 
Moreover we can choose $I$ so that $I_i= 0$ whenever $X_i =0$.
\end{lemma}
\begin{proof}
For each $i$ choose an $\I$-monomorphism $Z_i(X){} \rightarrow J_i$ with $J_i \in \I$
and let $X_i \rightarrow Q_i$ be the base
change of this $\I$-monomorphism along the inclusion
$Z_i(X) \hookrightarrow X_i$. Choose next an $\I$-monomorphism
$Q_i \rightarrow I_i$ with $I_i \in \I$ and define $X_i
\rightarrow I_i$ to be the composite $\I$-monomorphism 
$X_i \rightarrow Q_i \rightarrow I_i$.
Finally, consider the base change of $Q_i \rightarrow I_i$ along
$Q_i \rightarrow B_i(X)$. This is summarized in the following
diagram with exact rows:
\[
\xymatrix{
0 \ar[r] & Z_i(X) \ar[r] \ar@{>->}[d] & X_i \ar[r] \ar@{>->}[d] & B_i (X) \ar[r] \ar@{=}[d] & 0 \\
0 \ar[r] & J_i \ar[r] & Q_i \ar[r] \ar@{>->}[d] & B_i (X) \ar[r] \ar@{>->}[d] & 0 \\
& & I_i \ar[r] & R_i \ar[r] & 0.
}
\]
To define the boundary map $d_i  : I_i \rightarrow I_{i-1}$,
notice that since $B_i(X) \rightarrow R_i$ is an
$\I$-monomorphism and $J_{i-1}$ belongs to $\I$, 
the composite $B_i(X) \hookrightarrow Z_{i-1}(X) \rightarrow J_{i-1}$ admits a
factorization through $R_i$ and we get a map $R_i \rightarrow
J_{i-1}$. Define $d_i$ to be the composite $I_i \rightarrow R_i
\rightarrow J_{i-1} \rightarrow Q_{i-1} \rightarrow I_{i-1}$.
The composite $d_{i}d_{i+1}$ is zero since it factors through $J_i
\rightarrow Q_i \rightarrow B_i(X)$.
\end{proof}

\begin{lemma}\label{lem trivmaptrivcoffib}
For any object $M \in \ch(\A)_{\leq 0}$, the trivial map  $M \rightarrow 0$ can be factored
as a trivial cofibration followed by a fibration $M \stackrel{\sim}{\hookrightarrow} RM \twoheadrightarrow 0$.
\end{lemma}

\begin{proof}
By Lemma~\ref{lemma reschains} we can find a degreewise $\I$-monomorphism
${M  \rightarrow I_0}$ where the complex $I_0$ is made of $\I$-injectives. Let $K_0$ denote
the cokernel of this map and choose again a degreewise
$\I$-monomorphism $K  \rightarrow I_1$. Repeating the
process we construct a map from $M$ to a double complex $I_0 \rightarrow I_1 \rightarrow I_2 \rightarrow \dots$
made of $\I$-injectives $I_{\ast,\ast}$.
As a direct sum of $\I$-injectives is again an $\I$-injective the
total complex $\text{Tot}(I)_m = \oplus_{q-p = m} I_{p,q}$ is
fibrant. The induced map $M \rightarrow \text{Tot}(I)=RM$  is
level-wise the sum of the maps $M_m \rightarrow I_{0,m}$ 
and zero maps and thus is an $\I$-monomorphism.

By construction, for any $W \in \I$, the functor $\A(-,W)$
transforms the sequence $K_{p,q} \rightarrow I_{p+1,q}
\rightarrow K_{p+1,q} \rightarrow 0$ into an exact sequence. In
particular, applying $\A(-,W)$ to the double complex
$I_{p,q}$ yields a double complex which is acyclic in the
$p$-direction. The spectral sequence of the complex
$\A(\text{Tot}(I_{\ast,\ast}),W)= \text{Tot}(\A(I_{\ast,\ast},W))$
collapses thus on one line, which shows that the induced map
$M \rightarrow RM$ is an $\I$-equivalence.
\end{proof}

\begin{lemma}\label{lem trivmapcoftrivfib}
For any object $M \in \ch(\A)_{\leq 0}$, the trivial map $M \rightarrow 0$ can be factored as a cofibration
followed by a trivial fibration $M  \hookrightarrow P
\stackrel{\sim}{\twoheadrightarrow} 0$.
\end{lemma}
\begin{proof}
First factor $M \rightarrow 0$ as $M \hookrightarrow RM \twoheadrightarrow0$ by Lemma \ref{lem
trivmaptrivcoffib}. 
Then perform the cone construction to get a chain map $RM \rightarrow C(RM)$ which
is a cofibration to an acyclic complex by Lemma~\ref{lemma cone}.
Finally, $P=C(RM)$ is degrewise a sum of $\I$-injectives, hence a fibrant object. 
\end{proof}


We are now ready to prove the factorization axiom.

\begin{proposition}\label{prop coftrivfib}
Any map $M \rightarrow N$ can be factored as a cofibration
followed by an acyclic fibration.
\end{proposition}
\begin{proof}
First apply Lemma \ref{lem cylinder} to get a factorization $M
\stackrel{\sim}{\hookrightarrow} \hbox{\rm Cyl}(f) \rightarrow N$, where the map
$\hbox{\rm Cyl}(f) \rightarrow N$ is a split epimorphism in each degree. Let $K$
denote the kernel of $\hbox{\rm Cyl}(f) \rightarrow N$ and factor the trivial map
$K \rightarrow 0$ as  $K \rightarrow P \stackrel{\sim}{\twoheadrightarrow} 0$ 
by Lemma \ref{lem trivmapcoftrivfib}. 

Perform now the cobase change of $K \rightarrow \hbox{\rm Cyl}(f)$ along the cofibration 
$K \hookrightarrow P$, a situation we sum up in the following diagram:
\[
\xymatrix{ 
K \ar[dr] \ar@{^{(}->}[r] & P \ar[dr] && \\
& \hbox{\rm Cyl}(f) \ar@{^{(}->}[r] \ar[drr] &  X \ar@{->>}[dr]^\sim &\\
M \ar@{^{(}->}[ur]\ar[rrr]^f & & & N
}
\]
This yields a cofibration $\hbox{\rm Cyl}(f) \hookrightarrow X$.
The map $X \rightarrow N$ is induced by $\hbox{\rm Cyl}(f) \rightarrow N$
and the zero map $P \rightarrow N$; it is a split epimorphism 
since $\hbox{\rm Cyl}(f) \rightarrow N$ is so. Moreover its kernel
is the fibrant complex $P$ by construction. This complex is $\I$-trivial
so that $X \rightarrow N$ is a trivial $\I$-fibration.
\end{proof}

\begin{proposition}\label{prop trivcoffib}
Any map $M \rightarrow N$ can be factored as an acyclic
cofibration followed by a fibration.
\end{proposition}
\begin{proof}
As above, first apply Lemma \ref{lem cylinder} to factor 
$f: M \stackrel{\sim}{\hookrightarrow} \hbox{\rm Cyl}(f) \rightarrow N$, where we point
out that the first map is an acyclic cofibration. 
Consider the kernel $K$ of $\hbox{\rm Cyl}(f) \rightarrow N$, and
factor the map $K \rightarrow 0$ as in Lemma \ref{lem trivmaptrivcoffib}
$K \stackrel{\sim}{\hookrightarrow} RK  \twoheadrightarrow 0$. Perform next the
cobase change of $K \rightarrow RK$ along $K \rightarrow \hbox{\rm Cyl}(f)$.
Since cofibrations and weak equivalences are preserved under cobase
change we get an acyclic cofibration $\hbox{\rm Cyl}(f)  \stackrel{\sim}{\hookrightarrow} X$. 
We conclude just as in Proposition~\ref{prop coftrivfib} that the induced map $X \rightarrow N$ is an
$\I$-fibration.
%
\end{proof}

\begin{point}{\bf Lifting axioms.}
\label{point lifting}
We prove here the left lifting property for cofibrations with respect to trivial fibrations, and then the right lifting property
for fibrations with respect to trivial cofibrations.
\end{point}

\begin{lemma}\label{lem caractrivfib}
Let $p : E \rightarrow B$ be an acyclic fibration and denote
by $K$ its kernel. Then $E = K \oplus B$, $p$ is the second
projection and $K$ splits as a direct sum of complexes of the form
$\xymatrix@1{0 \ar[r] & W \ar@{=}[r] & W \ar[r] & 0}$ with $W  \in \I$.
\end{lemma}
\begin{proof}
As $K$ is $\I$-trivial and made of $\I$-injectives it splits as a
direct sum of complexes of the form $\xymatrix@1{0 \ar[r] & W
\ar@{=}[r] & W \ar[r] & 0}$ with $W  \in \I$. 
Such complexes are both projective and injective in the category of chain
complexes, therefore $E$ splits as 
$\ker p \oplus \textrm{im} \, p = K \oplus B$.
\end{proof}

\begin{proposition}\label{prop LLP}
Let $p : E \stackrel{\sim}{\twoheadrightarrow} B$ be an acyclic fibration and $X \hookrightarrow Y$ a cofibration.
In any commutative square
\[
\xymatrix{
X \ar[r] \ar@{^{(}->}[d]& E \ar@{->>}[d]^\sim \\
Y \ar@{.>}[ur] \ar[r] & B
}
\]
there is a dotted arrow making both triangles commutative.
\end{proposition}

\begin{proof}
As $E \rightarrow B$ is an acyclic fibration, the problem reduces 
by Lemma~\ref{lem caractrivfib} to find a lift $Y \rightarrow K$,
where $K$ is the kernel of $E \rightarrow B$, and hence of the form $\xymatrix@1{0 \ar[r] & K_{-i}  \ar@{=}[r] & K_{-i} \ar[r] & 0}$ 
with $K_{-i} \in \I$ for any $i \leq 0$.
In the following cube
\[
\xymatrix{
X_{-i} \ar[rr] \ar[dd] \ar[dr] & & K_{-i} \ar@{=}[dr] \ar[dd]|\hole & \\
& X_{-i-1} \ar@{->}[dd] \ar[rr] & & K_{-i} \ar[dd] \\
Y_{-i} \ar[rr]|\hole \ar[dr] & & 0 \ar[dr] \\
& Y_{-i-1} \ar[rr] \ar@{.>}[uurr]& & 0
}
\]
The lifting $h : Y_{-i-1} \rightarrow  K_{-i}$ exists because
$X_{-i-1} \rightarrow Y_{-i-1}$ is an $\I$-monomorphism and $K_{-i} \in \I$. 
Define $Y_{-i} \rightarrow K_{-i}$ as the composite
$Y_{-i} \rightarrow Y_{-i-1} \stackrel{h}{\rightarrow} K_{-i}$.
One easily checks that this gives a chain map $Y \rightarrow K$
with the desired properties.
\end{proof}

\begin{proposition}\label{prop RLP}
Let $p\colon E \twoheadrightarrow B$ be a fibration and $i\colon X \stackrel{\sim}{\hookrightarrow} Y$ a trivial cofibration.
In any commutative square
\[
\xymatrix{
X \ar[r]  \ar@{^{(}->}[d]_i^\sim& E \ar@{->>}[d]^p\\
Y \ar@{.>}[ur]_h \ar[r]_\ell & B
}
\]
there is a dotted arrow making both triangles commutative.
\end{proposition}

\begin{proof}
We define a lifting $h\colon Y \rightarrow E$ step by
step. Let $K$ be the kernel of the chain
map $p$. As $p$ is a fibration,  in each degree $E_n = B_n \oplus K_n$ and $p_n$ is the first
projection. Denote by $f_n$ the composite $X_n \rightarrow B_n \oplus K_n \rightarrow K_n$. To define the lift $h$ we only need to extend
the map $f_n\colon X_n \rightarrow K_n$ along $X_n \rightarrow Y_n$:
\[
\xymatrix{
X_n \ar[r]^{f_n}  \ar[d] & K_n \\
Y_n \ar@{.>}[ur]_{\exists k_n}
}
\]
in such a way that the map $h = (\ell,k)$ is a chain map. For this we proceed by induction on $n$. When $n=0$, observe that, since 
$K_0$ is $\I$-injective and $i$ is a cofibration, we have a quasi-isomorphism of cochain complexes
\[
\xymatrix{
0 & \text{Hom}(X_0,K_0) \ar[l] & \text{Hom}(X_{-1}, K_{0}) \ar[l] & \ar[l] \cdots \\
0 & \text{Hom}(Y_0,K_0) \ar[l]  \ar[u]^{i^*_0} & \text{Hom}(Y_{-1}, K_{0}) \ar[l] \ar@{->>}[u]^{i^*_{-1}} & \ar[l] \cdots
}
\]

In particular  there exists $\xi \in \text{Hom}(X_{-1},K_0)$ and $\phi \in \text{Hom}(Y_0,K_0)$ 
such that $f_0 + \xi \partial_X = i^*_{-1} \phi$. Since $X_{-1} \rightarrow Y_{-1}$ is an 
$\I$-monomorphism, there exists $\zeta\colon Y_{-1} \rightarrow K_0$ such that 
$\xi$ factors through $Y_{-1}$ as $X_{-1} \rightarrow Y_{-1} \stackrel{\zeta}{\rightarrow} K_0$.
Define $k_0 \colon Y_{0} \rightarrow K_0$ to be $\phi - \zeta \partial_Y$.
The desired lift $h \colon Y_0 \rightarrow B_0 \oplus K_0$ is then $l_0 \oplus
k_0$.

For $n\leq -1$, we assume that $k_{n+1}$ has been constructed.
The differential of the complex $E$ written according to the degree-wise splitting $E=B\oplus K$ has the form:
%
%
\[
\left( \begin{matrix}
\partial^{n+1}_B & 0 \\
\Delta_{n+1} & \partial_K^{n+1} 
\end{matrix}\right) : B_{n+1}\oplus K_{n+1} \longrightarrow  B_n \oplus K_n
\]
We also have a commutative diagram of solid arrows:
\[
\xymatrix{
 & X_{n+1} \ar[d]^{i_{n+1}} \ar[rr] \ar[ddl] &  & B_{n+1} \oplus K_{n+1} \ar[ddl] \\
 & Y_{n+1} \ar[urr]^{(\ell_{n+1},k_{n+1})}\ar[ddl]  & \\
 X_{n} \ar[rr] \ar[d]^{i_n} &   & B_{n} \oplus K_{n} &  \\
   Y_{n}  \ar@{.>}[urr]_{(\ell_n,k_n)}  &  
}
\]
Finally the trivial cofibration $i$ induces as above a quasi-isomorphism of cochain complexes:
\[
\xymatrix{
 \text{Hom}(X_{n+1},K_n)  & \text{Hom}(X_{n}, K_{n}) \ar[l] & \ar[l] \text{Hom}(X_{n-1},K_n)  \\
 \text{Hom}(Y_{n+1},K_{n})   \ar[u]^{i^*_{n+1}} & \text{Hom}(Y_{n}, K_{n}) \ar[l] \ar@{->>}[u]^{i^*_{n}} & \ar[l] \text{Hom}(Y_{n-1},K_n) \ar@{->>}[u]^{i^\ast_{n-1}}
}
\]
We are looking for a map $k_n$ that is firstly a chain map, and secondly extends $f_n$. This translates into the following equations:
\begin{equation}\label{eq:first}
\Delta_{n+1} \circ \ell_{n+1} + \partial_K^{n+1}k_{n+1} = k_n \circ \partial_Y^{n+1} 
\end{equation}
\begin{equation}\label{eq:second}
k_n \circ i_n = f_n 
\end{equation}
Observe that Equation~(\ref{eq:second}) expresses an equality in $\text{Hom}(X_n,K_n) $ while Equation~(\ref{eq:first}) 
is an equality in  $\text{Hom}(Y_{n+1},K_n)$. Precompose the later by $i_{n+1}: X_{n+1} \rightarrow Y_{n+1}$, to get:
\begin{equation}\label{eq:third}
\Delta_{n+1} \circ \ell_{n+1}\circ i_{n+1} + \partial_K^{n+1}k_{n+1} \circ i_{n+1} = k_n \circ \partial_Y^{n+1} \circ i_{n+1}
\end{equation}
By commutativity of the back face of the commutative diagram above, the left hand side of this equation is equal to
\[
f_n \circ \partial_X^{n+1}
\]
which is a trivial cocycle in $\text{Hom}(X_{n+1},K_n)$. 
Since $i^\ast$ is a quasi-isomorphism, this implies that the left hand side of Equation~(\ref{eq:first}) is a trivial cocycle in $\text{Hom}(Y_{n+1},K_n)$. 
In particular there is a map $\phi_n:Y_n \rightarrow K_n$ such that
\begin{equation*}
\Delta_{n+1} \circ \ell_{n+1} + \partial_K^{n+1}k_{n+1} = \phi_n \circ \partial_Y^{n+1} 
\end{equation*}
Since by construction $f_n \circ \partial^{n+1}_X = \phi_n\circ \partial^{n+1}_Y \circ i_{n+1} = \phi_n \circ i_n \circ \partial^{n+1}_X$,
there is a map $\zeta_{n-1} \in \text{Hom}(X_{n-1},K_{n}) $ such that $f_n = \phi_n \circ i_n + \zeta_{n-1}\circ \partial_X^n$.
By surjectivity of $i_{n-1}^*$ we may lift this map to $\xi_{n-1} \in \text{Hom}(Y_{n-1},K_n)$ such that $ \xi_{n-1} \circ i_{n-1} = \zeta_{n-1}$, 
and the map $k_n = \phi_n +   \xi_{n-1} \circ \partial_Y^n$ is the one we are looking for.
\end{proof}

\begin{theorem}\label{thm modstructureabove}
Assume that $\A$ has enough $\I$-injectives. Then the choice
of $\I$-weak equivalences, $\I$-cofibrations, and $\I$-fibrations gives 
$\ch_{\leq 0}(\A)$ the structure of a model category.
\end{theorem}

\begin{proof}
The category $Ch_{\leq 0}(\A)$ is clearly closed under both limits and colimits, which proves (MC1).
Since quasi-isomorphisms satisfy the ``2 out of 3" property so do $\I$-weak equivalences, this is (MC2).

Let us prove (MC3). Retracts of epimorphisms and of
quasi-isomorphisms are epimorphisms and quasi-isomorphisms
respectively, so $\I$-cofibrations and $\I$-weak equivalences are
preserved under retracts. As for $\I$-fibrations, notice that the retract of a map with a
section also has a section. Moreover since $\I = \overline{\I}$ is
stable under retracts we conclude that the retract of an
$\I$-fibration is again an $\I$-fibration. 

Finally, the factorization axiom (MC4) and the lifting axiom (MC5) have been established
in the preceding propositions.
\end{proof}

\section{Elementary algebra for topologists}
\label{sec:algebra}
This appendix contains a few elementary and well-known facts about localization, injective
envelopes, and local cohomology; none of these is new but we need it explicitly to describe relative resolutions in the case 
we restrict the notion of injectives. For a prime ideal $\mathfrak p$, we denote by
$M_{\mathfrak p}$ the localization of an $R$-module $M$ at $\mathfrak p$. The first lemma
will allow us to reduce certain problems to the case of a local ring, namely$R_{\mathfrak p}$.

\begin{lemma}
\label{lemma local-to-global}
An $R$-module $M$ is zero if and only if $M_{\mathfrak p}$ is zero for all prime
ideals~${\mathfrak p}$.
\end{lemma}
\begin{proof}
Let us assume that $M$ is non-zero, but $M_{\mathfrak p} = 0$ for any prime
ideal ${\mathfrak p}$. We choose a non-zero element $x \in M$ and consider its
annihilator. This ideal is contained in a maximal ideal ${\mathfrak m}$ and
since $M_{\mathfrak m} = 0$, there must exist an element $r \in R \setminus {\mathfrak m}$
such that $rx = 0$, a contradiction.
\end{proof}

A theorem of Matlis, \cite{MR0099360}, describes the injective
modules as direct sums of injective hulls $E(R/{\mathfrak p})$ of quotients of
the ring by prime ideals. The following two lemmas give some
properties of these indecomposable injective modules.

\begin{lemma}
\label{lemma local properties of injectives}
If ${\mathfrak q} \subset {\mathfrak p}$, the module $E(R/{\mathfrak q})$ is ${\mathfrak p}$-local, and otherwise
$E(R/{\mathfrak q})_{\mathfrak p} = 0$.
\end{lemma}

\begin{proof}
Assume ${\mathfrak q} \subset {\mathfrak p}$ and fix $r \not\in {\mathfrak p}$. The multiplication by
$r$ on $E(R/{\mathfrak q})$ is an isomorphism, so $E(R/{\mathfrak q})$ is ${\mathfrak p}$-local. Assume
now that ${\mathfrak q} \not\subset {\mathfrak p}$. Then ${\mathfrak q}^m \not\subset {\mathfrak p}$ for any $m \geq
1$. If $x$ is any element of $E(R/{\mathfrak q})$, its annihilator is ${\mathfrak q}^m$ for
some positive integer $m$ since $E(R/{\mathfrak q})$ is ${\mathfrak q}$-torsion. There
exists thus an element $s \in {\mathfrak q}^m$ which does not belong to ${\mathfrak p}$ and
such that $s x = 0$. Hence $x_{\mathfrak p} = 0$. This shows that $E(R/{\mathfrak q})_{\mathfrak p} = 0$.
\end{proof}

\begin{lemma}
\label{lemma hom between injectives}
The $R$-module of homomorphisms $\Hom_R(E(R/{\mathfrak p}), E(R/{\mathfrak q}))$ is non-zero
if and only if ${\mathfrak p} \subset {\mathfrak q}$.
\end{lemma}

\begin{proof}
Since $E(R/{\mathfrak q})$ is ${\mathfrak q}$-local by the previous lemma, any homomorphism
factors through the ${\mathfrak q}$-localization of $E(R/{\mathfrak p})$, which is zero
unless ${\mathfrak p} \subset {\mathfrak q}$. This proves one implication. 
In this case the quotient morphism $R/{\mathfrak p} \rightarrow
{\mathfrak q}$ extends to the injective envelopes showing the other implication.
\end{proof}

Let us now introduce local cohomology, a good reference for which is~\cite{MR2355715}.

\begin{definition}
\label{def local}
{\rm Given an ideal ${\mathfrak p}$ in $R$, the ${\mathfrak p}$-\emph{torsion} of an
$R$-module $M$ is the submodule $\Gamma_{\mathfrak p} (M)$ of elements with
annihilator ${\mathfrak p}^m$ for some positive integer $m$. The \emph{local
cohomology} modules $H_{\mathfrak p}^*(-)$ with support in ${\mathfrak p}$ are the right
derived functors of~$\Gamma_{\mathfrak p}$.}
\end{definition}

Explicitly, to compute the local cohomology of a module $M$, we
construct an injective resolution $I_\bullet$ of $M$ and compute
$H_{\mathfrak p}^j(M) = H^j(\Gamma_{\mathfrak p}(I_\bullet))$. Our last lemma helps us to
understand how this ${\mathfrak p}$-torsion injective complex look like.

\begin{lemma}
\label{lemma p-torsion injective}
The ${\mathfrak p}$-torsion module $\Gamma_{\mathfrak p}(E(R/{\mathfrak q})) = E(R/{\mathfrak q})$ 
if ${\mathfrak p} \subset {\mathfrak q}$ and is zero otherwise.
\end{lemma}

\begin{proof}
Again this follows from the fact that $E(R/{\mathfrak q})$ is ${\mathfrak q}$-torsion.
\end{proof}

\begin{remark}
\label{remark relative vs local}
{\rm Let $R$ be a local ring with maximal ideal ${\mathfrak m}$ and let us
consider the generization closed subset of $Spec(R)$ given by $S =
\{ {\mathfrak q} \, | \, {\mathfrak q} \neq {\mathfrak m} \}$. It yields the injective class $\I$
generated by all injective envelopes $E(R/{\mathfrak q})$ with ${\mathfrak q} \neq {\mathfrak m}$
see \cite[Proposition~3.1]{CPSinj}. 
Given a module $M$ and an injective resolution $I_\bullet$, we have a
triangle in the derived category $\Gamma_{\mathfrak m}(I_\bullet) \rightarrow
I_\bullet \rightarrow W_\bullet$, where $W_\bullet$ is an
$\I$-relative injective resolution of $M$. In particular
$H_k(W_\bullet) \cong H^{k+1}_{\mathfrak m}(M)$ for $k \geq 2$.}
\end{remark}

\begin{proposition}
\label{proposition Cech}
Let $R$ be a Noetherian ring and ${\mathfrak p}$ be the radical of $(x_1, \dots,
x_n)$. Then $H_{\mathfrak p}^k(M) = 0$ for any $k > n$ and any module $M$.
\end{proposition}

\begin{proof}
Since the torsion functor does not see the difference between an
ideal and its radical, we can assume that ${\mathfrak p} = (x_1, \dots, x_n)$.
Then the local cohomology can be computed by means of the \v{C}ech
complex $\otimes_i \check{C}(x_i, R) \otimes M$,
\cite[Theorem~7.13]{MR2355715}. Here $\check{C}(x, R)$ is the
complex $0 \rightarrow R \rightarrow R_{x} \rightarrow 0$
concentrated in degrees $0$ and $1$. The \v{C}ech complex is thus
concentrated in degrees $\leq n$.
\end{proof}

\begin{remark}
\label{remark uptonilpotence}
{\rm If $R$ is a Noetherian local ring of dimension~$d$, then the
maximal ideal can always be expressed as the radical of an ideal
generated by $n$ elements, see \cite[Theorem~1.17]{MR2355715}.}
\end{remark}


\bibliographystyle{amsplain}
\bibliography{homalg}
\providecommand{\bysame}{\leavevmode\hbox to3em{\hrulefill}\thinspace}
\providecommand{\MR}{\relax\ifhmode\unskip\space\fi MR }
\providecommand{\MRhref}[2]{%
  \href{http://www.ams.org/mathscinet-getitem?mr=#1}{#2}
}
\providecommand{\href}[2]{#2}

\end{document}